\documentclass[a4paper,10pt]{article}

\usepackage{amsmath,amsfonts,amssymb,latexsym,amsthm, enumerate, verbatim, a4wide}

  \theoremstyle{plain}
  \newtheorem{thm}{Theorem}[section]
  
  \newtheorem{prop}[thm]{Proposition}
  \newtheorem{lem}[thm]{Lemma}
  \newtheorem{defi}[thm]{Definition}
  
  \newtheorem{remark}[thm]{Remark}

\DeclareMathOperator{\tr}{Tr}

\DeclareMathOperator{\ric}{Ric}
\DeclareMathOperator{\diam}{Diam}
\DeclareMathOperator{\vol}{vol}

\DeclareMathOperator{\ve}{\varepsilon}
\DeclareMathOperator{\bino}{bin}

\DeclareMathOperator{\supp}{Supp}

\DeclareMathOperator{\degre}{deg}
\DeclareMathOperator{\EG}{ E \Gamma}
\DeclareMathOperator{\SEG}{SE\Gamma}
\DeclareMathOperator{\im}{im}
\DeclareMathOperator{\grad}{grad}

\title{$W_{1,+}$-interpolation of probability measures on graphs}

\author{Erwan Hillion \footnote{Department of Mathematics, University of Luxembourg, erwan.hillion@uni.lu}}

\begin{document}

\maketitle 

\begin{abstract}
We generalize an equation introduced by Benamou and Brenier in ~\cite{BBOriginal} and characterizing Wasserstein $W_p$-geodesics for $p>1$, from the continuous setting of probability distributions on a Riemannian manifold to the discrete setting of probability distributions on a general graph. \\
Given an initial and a final distributions $(f_0(x))_{x \in G}$, $(f_1(x))_{x \in G}$, we prove the existence of a curve $(f_t(k))_{t \in [0,1], k \in \mathbb{Z}}$ satisfying this Benamou-Brenier equation. We also show that such a curve can be described as a mixture of binomial distributions with respect to a coupling that is solution of a certain optimization problem. 
\end{abstract}

\section{Introduction}

Given some $p \geq 1$, we consider the space $\mathcal{P}_p(X)$ of probability distributions over a metric space $(X,d)$ having a finite $p$-th moment. On this space we define the Wasserstein distance $W_p$ by
\begin{equation}\label{eq:MKproblem}
W_p(\mu_0,\mu_1)^p := \inf_{\pi \in \Pi(\mu_0,\mu_1)} \int_{X \times X} d(x_0,x_1)^p d\pi(x_0,x_1),
\end{equation}
where the set $\Pi(\mu_0,\mu_1)$ is the set of couplings of $\mu_0$ and $\mu_1$, i.e. the set of probability distributions $\pi$ on $X \times X$ having $\mu_0$ and $\mu_1$ as marginals. 

\medskip

An comprehensive study of the minimization problem~\eqref{eq:MKproblem}, called Monge-Kantorovitch problem, can be found in Villani's textbooks~\cite{VillaniBook1} and \cite{VillaniBook2}. Let us recall what is important for our purposes: under very general assumptions, it is possible to prove the existence of a minimizer $\pi \in \Pi(\mu_0,\mu_1)$ for problem~\eqref{eq:MKproblem}, called optimal coupling, and that $W_p$ is indeed a metric on $\mathcal{P}_p(X)$. Moreover, if we suppose that $(X,d)$ is a geodesic space, i.e. if the distance $d(x_0,x_1)$ is exactly the length of the shortest curve joining $x_0$ to $x_1$, then the metric space $(\mathcal{P}_p(X),W_p)$ is also a geodesic space. In particular, each couple $\mu_0,\mu_1 \in \mathcal{P}_2(X)$ can be joined curve $(\mu_t)_{t \in [0,1]}$ of minimal length for $W_2$, called $W_2$-Wasserstein geodesic. 

\medskip

In their seminal papers \cite{SturmRicci01}, \cite{SturmRicci02} and \cite{LottVillani}, Sturm and independently Lott and Villani studied the links between the geometry of a measured geodesic space $(X,d,\nu)$ and the behaviour of the entropy functional along the $W_2$-Wasserstein geodesics on $\mathcal{P}_2(X)$. For instance, $(X,d,\nu)$ is said to satisfy the curvature condition $CD(K,\infty)$ for some $K \in \mathbb{R}$ if for each couple of probability distributions $\mu_0,\mu_1 \in \mathcal{P}_2(X)$ there exists a $W_2$-geodesic $(\mu_t)_{t \in [0,1]}$ such that
\begin{equation} \label{eq:CDKinfty}
\forall t \in [0,1] \ ,\ H_\nu(\mu_t) \leq (1-t) H_\nu(\mu_0)+ tH_\nu(\mu_1) - K\frac{t(1-t)}{2} W_2(\mu_0,\mu_1)^2,
\end{equation}
where the relative entropy functional $H_\nu(\cdot)$ is defined by 
\begin{equation} \label{eq:DefEntropyCont}
H_\nu(\rho \nu) := \int_X \rho(x) \log(\rho(x)) d\nu(x)
\end{equation} 
if $\mu = \rho \nu$ for some density $\rho$, and by $H_\nu(\mu):=\infty$ otherwise.

\medskip

If the measured geodesic space $(X,d,\nu)$ is a compact Riemannian manifold with its usual distance an normalized volume measure, the curvature condition $CD(K,\infty)$ is shown to be equivalent to the bound $\ric \geq K$ on the Ricci curvature tensor. Another important property is the stability of the condition $CD(K,\infty)$ under measured Gromov-Hausdorff convergence. 

\medskip

Moreover, if $CD(K,\infty)$ is satisfied for some $K > 0$, one can prove functional inequalities on $(X,d,\nu)$ such as the logarithmic Sobolev inequality, which asserts that
\begin{equation} \label{eq:LSIcont}
H_\nu(f d\nu) \leq \frac{1}{2K} \int_X \frac{|\nabla^{-} f|^2}{f} d\nu, 
\end{equation}
for any Lipschitz probability density $f$ and where $|\nabla^{-} f|$ is to be seen as a particular form of the norm of a gradient. As a corollary, it can be shown that under the condition $CD(K,\infty)$ for $K > 0$ a Poincar\'e inequality holds: for any Lipschitz funtion $h: X \rightarrow \mathbb{R}$ such that $\int_X h d\nu =0$, we have
\begin{equation}
\int_X h^2 d\nu \leq \frac{1}{2K} \int_X |\nabla^{-} h|^2 d\nu.  
\end{equation}

Since the pioneering works of Sturm and Lott-Villani, the theory of measured geodesic spaces satisfying $CD(K,\infty)$ has been thoroughly studied in a large number of papers, among which the most impressive are the works by Ambrosio, Gigli and Savar\'e (see for instance~\cite{AmbrosioGigliSavare}) and by Erbar, Kuwada and Sturm (\cite{ErbarKuwadaSturm}).

\medskip

Several obstacles prevent us from a direct generalization of Sturm-Lott-Villani theory to the framework of discrete metric spaces. Indeed, if $(X,d)$ is a graph with its usual distance, equation~\eqref{eq:MKproblem} still defines a metric on the space $\mathcal{P}_p(X)$, but if $p > 1$ then the length of non-trivial curves in $(\mathcal{P}(X),W_p)$ is $+ \infty$, which means that it is not a geodesic space. In particular, Wasserstein $W_2$-geodesics do not exist in general. 

\medskip

Several solutions have been proposed to overcome this difficulty, and there are now many different definitions of Ricci curvature bounds on discrete spaces. The most notable of them are the coarse Ricci curvature, defined by Ollivier in~\cite{OllivierRicci}, and the Erbar-Maas curvature, defined in~\cite{ErbarMaas}. The latter is based on the study of the gradient flow of the entropy and present some similarities with our own approach.

\medskip

In this paper, we place ourselves in the framework of a connected and locally finite graph $G$, endowed with its usual graph distance and the counting measure as the reference measure. In this framework, a probability distribution will be denoted by its density, i.e. by a function $f : G \rightarrow \mathbb{R}_+$ suwh that $\sum_{x \in G}f(x)=1$. Given two probability distributions $f_0$ and $f_1$ on $G$, we investigate the question of the generalization of the notion of $W_p$-geodesic joining $f_0$ to $f_1$ in a setting where such a curve does not exist. Our goal is to provide a way to chose, among the set of all $W_1$-geodesics joining $f_0$ to $f_1$, a curve which shares some properties satisfied by $W_p$-geodesics for $p >1$. Such curves will be called $W_{1,+}$ geodesics on the graph $G$.

\medskip

This article is to be seen as the first of a two-paper research work. A following article will investigate the convexity properties of the entropy functional along those particular $W_1$ geodesics, in the view of obtaining a discrete version of equation~\eqref{eq:CDKinfty} strong enough to imply discrete versions of log-Sobolev or Poincar\'e inequalities. This ultimate goal has to be kept in mind even in this present paper because it will motivate the definition of a $W_{1,+}$-geodesic between $f_0$ and $f_1$: along such a curve, some technical tools will allow us to give bounds on the second derivative of the entropy.

\medskip

Our starting point is the article~\cite{BenamouBrenier}, by Benamou and Brenier. In this paper, the authors reformulate the Monge-Kantorovitch problem in terms of velocity fields and prove the following:

\begin{thm} \label{th:WpBB}
Let $\mu_0, \mu_1$ be two probability distributions on a Riemannian manifold $(M,g)$ and $p > 1$. Then
\begin{equation} \label{eq:BBproblem}
W_p(\mu_0,\mu_1)^p = \inf \int_M \int_0^1 |v_t(x)|^p d\mu_t(x) dt,
\end{equation}
the infimum being taken over the families of probability distributions $(\mu_t) := (f_t d\vol)$ joining $\mu_0$ to $\mu_1$ and all velocity fields $(v_t(x))$ satisfying
\begin{equation*}\label{eq:TransportEquationCont}
\frac{\partial}{\partial t} f_t(x) = - \nabla \cdot (f_t(x) v_t(x)),
\end{equation*}
where $\nabla \cdot$ is the divergence operator on $M$. Moreover the minimizing curve $(\mu_t)_{t \in [0,1]}$ is the $W_p$-geodesic joining $\mu_0$ to $\mu_1$.
\end{thm}

This theorem has been extended to the framework of separable Hilbert spaces by Ambrosio, Gigli and Savr\'e in~\cite{AmbrosioGigliSavareBook}.

\medskip

The strategy used by Erbar and Maas in~\cite{ErbarMaas} is based on a generalization of the minimization problem~\eqref{eq:BBproblem} in the framework of discrete Markov chains. Our approach will consist in defining a discrete version of a characterization of its solutions. More precisely, as pointed in~\cite{BenamouBrenier}, the formal optimality condition for the optimization problem~\eqref{eq:BBproblem} can be written:

\begin{equation}\label{eq:velocityCont}
\frac{\partial}{\partial t} v_t(x) = - v_t(x) \nabla \cdot v_t(x).
\end{equation} 

Another point of view on the formal optimality condition~\eqref{eq:velocityCont} is provided by writing the velocity field $(v_t(x))$ as the gradient of a family of convex functions $v_t := \grad \Phi_t$. As explained for instance in~\cite{OttoVillani}, it can be proven that such a function $\Phi$ satisfies the Hamilton-Jacobi equation
\begin{equation}\label{eq:RiemannianHJ}
\frac{\partial}{\partial t} \Phi_t + \frac{1}{2} |\nabla \Phi_t|^2 = 0.
\end{equation}
It suffices to consider the gradient of equation~\eqref{eq:RiemannianHJ} to recover equation~\eqref{eq:velocityCont}.

\medskip

The links between the convexity of the entropy $H(t)$ of $\mu_t$ and the Ricci curvature tensor on the manifold $M$ are seen on the following heuristic formula, established by Otto and Villani in~\cite{OttoVillani}:
\begin{equation}\label{eq:OttoVillaniEntropy}
H''(t) = \int_M [\tr((D^2\Phi_t)^TD^2\Phi_t)+ \ric(\grad \Phi_t,\grad \Phi_t)] d\mu_t.
\end{equation}
In particular, the non-negativity of the tensor $\ric$ easily implies that $H''(t) \geq 0$.

\medskip

The formal optimality condition~\eqref{eq:velocityCont} on velocity fields makes sense only when $v$ is regular enough. The question of the regularity of optimal couplings is a difficult topic, see for instance~\cite{AmbrosioGigliSavareBook}. However, what is important for our purposes is that~\eqref{eq:velocityCont} can be used to construct $W_2$-geodesics: if $(f_t(x))$ is a smooth family of probability densities satisfying the transport equation~\eqref{eq:TransportEquationCont} for a smooth velocity field $(v_t(x))$ satisfying the condition~\eqref{eq:velocityCont}, then the curve $(f_t(x))$ is a $W_2$-geodesic.

\medskip

In the simpler framework of the real line $\mathbb{R}$ with usual distance and Lebesgue reference measure, it is possible to give an equivalent statement of this result without introducing explicitly the velocity field.:

\begin{prop} \label{prop:BBCondCont}
Let $(f_t(x))$ be a family of smooth probability densitites on $\mathbb{R}$. We define the families of functions 
\begin{equation}\label{eq:defghCont}
g_t(x) := -\int_{-\infty}^x \frac{\partial}{\partial t} f_t(z) dz \ , \ h_t(x) := -\int_{-\infty}^x \frac{\partial}{\partial t} g_t(z) dz .
\end{equation}
We suppose that $g_t(z)>0$ and that the following one-dimensional Benamou-Brenier condition holds:
\begin{equation}\label{eq:BBCondCont}
f_t(x) h_t(x) = g_t(x)^2. 
\end{equation}
Then $(f_t(x))$ is a $W_p$-geodesic for any $p >1$.
\end{prop}

To prove Proposition~\ref{prop:BBCondCont}, it suffices to realize that equation~\eqref{eq:BBCondCont} easily implies equation~\eqref{eq:velocityCont}.

\medskip

Apart from regularity issues, which will not play an important role in a discrete framework, the main restriction made in the statement of Proposition~\ref{prop:BBCondCont} is the non-degeneracy condition $g_t(z)>0$. It is quite easy to prove that such a condition implies that $f_0$ is stochastically dominated by $f_1$. In the setting of graphs, we will introduce the notion of $W_1$-orientation (see Paragraph~\ref{sec:W1orient}) in order to force the function $g_t$ to stay positive.

\medskip

The main purpose of this article is to study curves in the space of probability distributions on a graph which satisfy a discrete version of the Benamou-Brenier condition~\eqref{eq:BBCondCont}.

\begin{itemize}
\item The goal of Section~\ref{sec:Definition} is to provide a generalization of equations~\eqref{eq:defghCont} and \eqref{eq:BBCondCont} to this discrete setting. We will first show that these equations can be recovered in a particular form in the case of contraction of measures. Given a couple of probability distributions $f_0, f_1$ defined on $G$, we then endow $G$ with an orientation which will allow us to give a general definition of $W_{1,+}$-geodesics on $G$. The terminology ``$W_{1,+}$-geodesic'' will be explained by considering a discrete version of problem~\eqref{eq:BBproblem} when $p>1$ is close to $1$. 
\item In Section~\ref{sec:Structure}, we are looking for necessary conditions satisfied by $W_{1,+}$-geodesics on $G$. In particular, we prove in Theorem~\ref{thm:BBCBinom} that if $f_0$ and $f_1$ are finitely supported, then any $W_{1,+}$-geodesic $(f_t)$ can be written as a mixture of binomial distributions supported on geodesics of $G$.
\item In Section~\ref{sec:Existence} we prove the existence of $W_{1,+}$-geodesics $(f_t)$ with prescribed initial and final distributions $f_0$ and $f_1$. The construction of such curves suggests us strong links with the ``Entropic Interpolations'' studied in a recent series of papers by L\'eonard.
\end{itemize}

\section{The discrete Benamou-Brenier condition}\label{sec:Definition}

In this paper, we consider a locally finite and connected graph $G$. A path $\gamma$ on $G$ of length $n=L(\gamma)$ is a collection of vertices $\gamma(0),\ldots, \gamma(n) \in G$ such that $\gamma(i) \sim \gamma(i+1)$ for every $i=0,\ldots n-1$, where the relation $x \sim y$ means that $(xy)$ is in the edge set of the graph $G$. To any path $\gamma: \{0,\ldots n\} \rightarrow G$ are associated its endpoints $e_0(\gamma):= \gamma(0)$ and $e_1(\gamma) := \gamma(n)$.

\medskip

We will use the usual graph distance on $G$: $d(x,y)$ is the length of the shortest path joining $x$ to $y$. The set of geodesics joining $x$ to $y$, denoted by $\Gamma_{x,y}$, is the set of paths $\gamma$ joining $x$ to $y$ such that $L(\gamma)= d(x,y)$. The set of all geodesics on $G$ is denoted by $\Gamma(G)$.

\medskip

A coupling $\pi \in \Pi(f_0,f_1)$ is said to be a $W_p$-optimal coupling for some $p \geq 1$ if it is a minimizer for the functional
\begin{equation}
I_p : \pi \rightarrow \sum_{x,y \in G \times G} d(x,y)^p \pi(x,y).
\end{equation}
We denote by $\Pi_p(f_0,f_1)$ the set of $W_p$-optimal couplings.

\medskip

\begin{remark} 

The equality $I_1(\alpha \pi_1 + (1-\alpha) \pi_2) = \alpha I_1(\pi_1) +(1-\alpha) I_1(\pi_2)$ proves that the set $\Pi_1(f_0,f_1)$ is a convex subset of $\Pi(f_0,f_1)$. 
\end{remark}

\subsection{Contraction of measures and the Benamou-Brenier equation}

Among early attempts to generalize particular Wasserstein geodesics to the discrete case, one important example is given by the thinning operation:

\begin{defi}
Let $f$ be a probability distribution finitely supported on $\mathbb{Z}_+$. The thinning of $f$ is the family $(T_t f)$ of probability distributions defined by 
\begin{equation}
T_t f(k) := \sum_{l \geq 0} \bino_{l,t}(k) f(l) = \sum_{l \geq 0} \binom{l}{k} t^k (1-t)^{l-k} f(l),
\end{equation}
where by convention $\binom{l}{k}=0$ if $l < 0$ or if $k \notin \{0,\ldots l\}$.
\end{defi}

The operation $f \mapsto T_t f$ is often seen as a discrete version of the operation 
\begin{equation} \label{eq:thinningR} 
f(x) \mapsto f_t(x) := \frac{1}{t} f\left(\frac{x}{t}\right),
\end{equation} 
and is for instance used to state a weak law of small numbers (see~\cite{JohnsonSmallNumbers}) about the limit in distribution of $T_{1/n} (f^{\star n})$ when $n \rightarrow  \infty$. 

\medskip

We know that, given a smooth probability density $f$ on $\mathbb{R}$, the family $(f_t)$ defined by equation~\eqref{eq:thinningR} is a $W_p$-geodesic for any $p \geq 1$. According to Sturm-Lott-Villani theory, the metric space $(\mathbb{R},|\cdot |)$ satisfies the condition $CD(0,\infty)$, so the entropy $H(t)$ of $f_t$ with respect to the Lebesgue measure is a convex function of $t$. On the other hand, a theorem by Johnson and Yu (see~\cite{JohnsonEntropyThinning}) asserts that the entropy of the thinning $T_t f$ is also a convex function of $t$. The proof of this theorem given in~\cite{HillionContraction} relies on the following:

\begin{prop}\label{prop:BBCthinning}
Let $(f_t) := (T_t f)$ be the thinning family associated to a probability distribution $f = f_1$ supported on $\mathbb{Z}_+$. We define the families of functions $(g_t)$ and $(h_t)$ by
\begin{equation}\label{eq:defgh}
g_t(k) := -\sum_{l \leq k} \frac{\partial}{\partial t}f_t(l) \ , \ h_t(k) := -\sum_{l \leq k} \frac{\partial}{\partial t}g_t(l).
\end{equation}
The triple $(f_t,g_t,h_t)$ then satisfies the discrete Benamou-Brenier equation:
\begin{equation}\label{eq:BBCondZ}
\forall k \in \mathbb{Z} \ , \ f_t(k) h_t(k-1) = g_t(k) g_t(k-1).
\end{equation}
Moreover, $g_t(k) \geq 0$, and if $g_t(k)=0$ then either $f_t(k+1)=0$ or $h_t(k)=0$.
\end{prop}

\begin{remark} Denoting by $\nabla_1$ (resp. $\nabla_2$) the left derivative operator (resp. the left second derivative operator) defined by $\nabla_1 u(k) := u(k)-u(k-1)$ (resp. $\nabla_2 u(k) = u(k)-2u(k-1)+u(k-2)$), we thus have
\begin{equation*}
\frac{\partial f_t(k)}{\partial t} = -\nabla_1 g_t(k) \ ,\ \frac{\partial^2 f_t(k)}{\partial t^2} = \nabla_2 h_t(k).
\end{equation*}
\end{remark}

The proof of the convexity of the entropy along thinning families relies so importantly on Proposition~\ref{prop:BBCthinning} that this proof can be used \textit{verbatim} to prove a stronger statement:

\begin{prop} \label{thm:MBBCEntropy}
Let $(f_t)$ be a family of finitely supported probability distributions on $\mathbb{Z}$. We suppose that the families of functions $(g_t)$ and $(h_t)$, defined by equation~\eqref{eq:defgh}, satisfy the discrete Benamou-Brenier equation~\eqref{eq:BBCondZ} and the non-negativity condition $g_t(k) \geq 0$. Then the entropy $H(t)$ of $f_t$ is a convex function of $t$.
\end{prop}

Because the similarities with equation~\eqref{eq:BBCondCont}, it seems legitimate to consider a family of measures satisfying equation~\eqref{eq:BBCondZ} and the non-negativity condition $g_t(k) \geq 0$ as a pseudo $W_p$-geodesic, for $p > 1$, along which the entropy functional is convex, which is reminiscent of Sturm-Lott-Villani theory.

\medskip

The notion of thinning has been extended in~\cite{HillionContraction} to the setting of general graphs in the following way: we consider a probability distribution $f_1$ defined on $G$ and another probability measure $f_0$ which is a Dirac mass at a given point $o \in G$. In this case, an interpolating curve $(f_t)$, called contraction of $f_1$ on $o$, is defined as a mixture of binomial distributions by
\begin{equation}
f_t := \sum_{z \in G} \frac{1}{|\Gamma_{o,z}|} \sum_{\gamma \in \Gamma(o,z)} \bino_{\gamma,t},
\end{equation}
where the binomial distribution on $\gamma$ is related to the classical binomial distribution by
\begin{equation}
\forall p \in \{0,\ldots L(\gamma)\} \ , \ \bino_{\gamma,t}(\gamma(p)) := \bino_{L(\gamma),t}(p)
\end{equation}
and where $|\Gamma_{o,z}|$ denotes the cardinality of the set $\Gamma_{o,z}$ of geodesics joining $o$ to $z$.

\medskip

A couple of initial and final distributions $\delta_{o}=f_0$ and $f_1$ being given, we define a partial order on the set of vertices of $G$ by writing $x_1 \leq x_2$ if the vertex $x_1$ belongs to a geodesic $\gamma \in \Gamma_{o,x_2}$. If $x_1 \sim x_2$ and $x_1 \leq x_2$, we say that $(x_1x_2)$ is an oriented edge and we write $x_2 \in \mathcal{F}(x_1)$, $x_1 \in \mathcal{E}(x_2)$ or $x_1 \rightarrow x_2$.

\medskip

To the oriented graph $(G,\rightarrow)$ are associated two other oriented graphs:

\begin{defi}
The oriented edge graph $(E(G),\rightarrow)$ is the graph of oriented couples $x_1 \rightarrow x_2 \in G$, oriented itself by the relation $(x_1x_2) \rightarrow (x_2x_3)$. In particular, for any $(x_1x_2) \in (E(G),\rightarrow)$ we have 
\begin{equation*}
\mathcal{E}((x_1x_2)) = \{ (x_0x_1) : x_0 \in \mathcal{E}(x_1) \} \ , \ \mathcal{F}((x_1x_2)) = \{ (x_2x_3) : x_3 \in \mathcal{F}(x_2) \}.
\end{equation*}
Similarly, we define the graph of oriented triples $(T(G),\rightarrow) := (E(E(G)),\rightarrow)$, having as vertices the triples $(x_1x_2x_3)$ with $x_1 \rightarrow x_2 \rightarrow x_3$ and edges between each couple $(x_0x_1x_2)$ and $(x_1x_2x_3)$.
\end{defi}      

\begin{remark}
The graph $G$ being now oriented, the notations $E(G)$ and $T(G)$ stand for $(E(G),\rightarrow)$ and $(T(G),\rightarrow)$, which is a slight abuse of notation. For instance, $(xy) \in E(G)$ imply that $x \rightarrow y$. This remark will still be valid once introduced the $W_1$-orientation on $G$.
\end{remark}

Orienting the graph $G$ allows us to define a divergence operator:

\begin{defi}\label{def:DivergenceOriented}
The divergence of a function $g : E(G) \rightarrow \mathbb{R}$ is the function $\nabla g : G \rightarrow \mathbb{R}$ defined by $$\forall x_1 \in G \ , \ \nabla g(x_1) := \sum_{x_2 \in \mathcal{F}(x_1)} g(x_1x_2) - \sum_{x_0 \in \mathcal{E}(x_1)} g(x_0x_1).$$
Similarly, the divergence of a function $h : T(G) \rightarrow \mathbb{R}$ is the function $\nabla h : E(G) \rightarrow \mathbb{R}$ defined by $$\nabla  h(x_1x_2) := \sum_{(x_2x_3) \in \mathcal{F}(x_1x_2)} h(x_1x_2x_3)-\sum_{(x_0x_1) \in \mathcal{E}(x_1x_2)} h(x_0x_1x_2).$$
\end{defi}

\medskip

We use this orientation to express the function $f_t$ as a product of two functions satisfying interesting differential equations:

\begin{prop}\label{prop:ftutvt}
There exists a couple $(P_t), (Q_t)$ of families of non-negative functions on $G$ such that:
\begin{enumerate}
\item We have $f_t(x)=P_t(x)Q_t(x)$.
\item The functions $P$ and $Q$ satisfy the equations
\begin{equation}
\frac{\partial}{\partial t} P_t(x_1) = \sum_{x_0 \in \mathcal{E}(x_1)} P_t(x_0) \ , \ \frac{\partial}{\partial t} Q_t(x_1) = - \sum_{x_2 \in \mathcal{F}(x_1)} Q_t(x_2).
\end{equation}
\end{enumerate}
\end{prop}

This proposition is proven in~\cite{HillionContraction}. We can now use Definition~\ref{def:DivergenceOriented} and~\ref{prop:ftutvt} to state a generalized version of Proposition~\ref{prop:BBCthinning}:

\begin{prop} \label{prop:BBContraction}
We define the families of functions $(g_t) : E(G) \rightarrow \mathbb{R}$ and $(h_t) : T(G) \rightarrow \mathbb{R}$ by
\begin{equation}
g_t(x_1x_2) := P_t(x_1)Q_t(x_2) \ , \ h_t(x_0x_1x_2) := P_t(x_0) Q_t(x_2).
\end{equation}
\begin{enumerate}
\item The functions $f$, $g$ and $h$ satisfy the differential equations
\begin{equation} \label{eq:fghContraction}
\frac{\partial}{\partial t} f_t(x_1) = - \nabla g_t(x_1) \ , \  \frac{\partial}{\partial t} g_t(x_1x_2) = - \nabla h_t(x_1x_2).
\end{equation}
\item For every oriented triple $(x_0x_1x_2) \in T(G)$ we have
\begin{equation}\label{eq:BBGraph}
h_t(x_0x_1x_2) f_t(x_1) = g_t(x_0x_1) g_t(x_1x_2).
\end{equation}
\end{enumerate}
\end{prop}

\begin{remark} As in the thinning case, Proposition~\ref{prop:BBContraction}, and in particular equation~\eqref{eq:BBGraph} are used to study the convexity of the entropy functional along contraction families on graphs. 
\end{remark}

\subsection{The $W_1$-orientation} \label{sec:W1orient}

It is not possible to use directly Proposition~\ref{prop:BBContraction} to propose a general Benamou-Brenier condition because such a definition relies on an orientation of the graph $G$ which has been constructed by using the fact that $f_0$ is Dirac. As a first necessary step in the construction of general $W_{1,+}$-geodesics, we thus need to find a nice orientation on $G$, depending on the initial and final measures $f_0$ and $f_1$. 

\medskip

The term ``nice orientation'' is vague, but the study of the thinning and of the contraction families suggests that, in order to have interesting consequences on the convexity of the entropy, we should at least require that $g_t(x_1x_2) \geq 0$ for every $x_1 \rightarrow x_2 \in E(G)$. As we will see at the end of this paragraph, this requirement can be interpreted in the framework of optimal transportation theory. 

\medskip

We first recall some properties of supports of $W_1$-optimal couplings:

\begin{defi}
Given a couple $f_0,f_1$ of finitely supported measures, we associate the set
\begin{equation}
\mathcal{C}(f_0,f_1) := \{(x,y) \in G \times G \ : \ \exists \pi \in \Pi_1(f_0,f_1), \pi(x,y) > 0 \}.
\end{equation}
Equivalently, $\mathcal{C}(f_0,f_1)$ is the smallest subset of $G \times G$ containing the supports of all the $W_1$-optimal couplings between $f_0$ and $f_1$.
\end{defi}

\begin{prop}\label{prop:pifullsupp}
There exists $\pi \in \Pi_1(f_0,f_1)$ such that $\supp(\pi) = \mathcal{C}(f_0,f_1)$.
\end{prop}

\begin{proof} For every $(x,y) \in \mathcal{C}(f_0,f_1)$, there exists a coupling $\pi_{(x,y)} \in \Pi_1(f_0,f_1)$ with $\pi_{(x,y)}(x,y)>0$. As $f_0$ and $f_1$ are finitely supported, we can consider the barycenter 
\begin{equation}
\pi := \frac{1}{|\mathcal{C}(f_0,f_1)|} \sum_{(x,y) \in \mathcal{C}(f_0,f_1)} \pi_{(x,y)},
\end{equation}
which by convexity is in $\Pi_1(f_0,f_1)$ and which is clearly fully supported in $\mathcal{C}(f_0,f_1)$.
\end{proof}

\medskip

A tool often used when studying the support of optimal couplings is the cyclic monotonicity property:

\begin{lem} \label{lem:suppC}
If $(x_0,y_0), \ldots (x_p,y_p)$ are in $\mathcal{C}(f_0,f_1)$ then 
\begin{equation}\label{eq:suppC}
\sum_{i=0}^p d(x_i,y_i) \leq d(x_0,y_p) + \sum_{i=0}^{p-1} d(x_{i+1},y_i).
\end{equation}
\end{lem}

\begin{proof} We consider a coupling $\pi \in \Pi_1(f_0,f_1)$ as constructed in Proposition~\ref{prop:pifullsupp} and a number $0<a< \inf_i(\pi(x_i,y_i)).$ We introduce the function $h$ on $G \times G$ defined by $h(x_0,y_p) := a$, $h(x,y) := a$ if $(x,y)=(x_{i+1},y_i)$ for some $i \in \{0,\ldots p-1\}$, $h(x,y) := -a$ if $(x,y)=(x_i,y_i)$ for some $i \in \{0,\ldots p\}$,  and $h(x,y):=0$ elsewhere. Then $\pi + h$ is a coupling in $\Pi(f_0,f_1)$ and 
\begin{equation}
I_1(\pi+h)-I_1(\pi) = a \left( d(x_0,y_p) + \sum_{i=0}^{p-1} d(x_{i+1},y_i)-\sum_{i=0}^p d(x_i,y_i)\right), 
\end{equation}
but $ \pi  \in \Pi_1(f_0,f_1)$ implies $I_1(\pi+h) \geq I_1(\pi)$, which shows equation~\eqref{eq:suppC}.
\end{proof}

Lemma~\ref{lem:suppC} is used to define unambiguously an orientation on some edges of $G$:

\begin{thm}\label{th:OrientW1}
Let $\gamma^{(1)}, \gamma^{(2)}$ be two geodesics in $G$ such that $(e_0(\gamma^{(i)}),e_1(\gamma^{(i)})) \in \mathcal{C}(f_0,f_1)$ for $i=1,2$. Then for any $(k_1,k_2)$ with $k_i \leq L(\gamma^{(i)})-1$ we cannot have both identities $\gamma^{(1)}(k_1)=\gamma^{(2)}(k_2+1)$ and $\gamma^{(1)}(k_1+1) = \gamma^{(2)}(k_2)$.
\end{thm}

\begin{proof} Suppose that both identities $\gamma^{(1)}(k_1)=\gamma^{(2)}(k_2+1)$ and $\gamma^{(1)}(k_1+1) = \gamma^{(2)}(k_2)$ hold. By considering the path $e_0(\gamma^{(1)}),\cdots \gamma^{(1)}(k_1),\gamma^{(2)}(k_2+2),\cdots e_1(\gamma^{(2)})$, we see that $$d(e_0(\gamma^{(1)}),e_1(\gamma^{(2)})) \leq k_1+L(\gamma^{(2)})-k_2-1.$$ Similarly we have $$d(e_0(\gamma^{(2)}),e_1(\gamma^{(1)})) \leq k_2+L(\gamma^{(1)})-k_1-1.$$ Since $L(\gamma^{(i)})=d(e_0(\gamma^{(1)}),e_1(\gamma^{(1)}))$ for $i=1,2$, we have 
\begin{equation}
d(e_0(\gamma^{(1)}),e_1(\gamma^{(1)}))+d(e_0(\gamma^{(2)}),e_1(\gamma^{(2)})) \geq d(e_0(\gamma^{(1)}),e_1(\gamma^{(2)}))+d(e_0(\gamma^{(2)}),e_1(\gamma^{(1)}))+2,
\end{equation}
which by Lemma~\ref{lem:suppC} is a contradiction.
\end{proof}

\begin{defi} Let $f_0,f_1$ be two finitely supported probability distributions on $G$.
\begin{itemize}
\item The $W_1$-orientation with respect to $f_0,f_1$ is defined orienting the edge $(x,y) \in E(G)$ by $x \rightarrow y$ if there exists a geodesic $\gamma$ on $G$ such that
\begin{enumerate}
\item $(e_0(\gamma),e_1(\gamma)) \in \mathcal{C}(f_0,f_1)$.
\item $\gamma(k)=x$, $\gamma(k+1)=y$ for some $k \in \{0,\ldots L(\gamma)-1\}$.
\end{enumerate}
\item An oriented path on the oriented graph $(G,\rightarrow)$ is an application $\gamma:\{0,\ldots L\} \rightarrow G$ such that $\gamma(i) \rightarrow \gamma(i+1)$ for $i=0,\ldots L-1$. 
\item We define a partial order relation on the vertices of $G$ by writing $x \leq y$ if there exists an oriented path joining $x$ to $y$.
\end{itemize}
\end{defi}

An important property of the $W_1$-orientation is the following:

\begin{thm}\label{prop:GeodG}
Every oriented path on $(G,\rightarrow)$ is a geodesic.
\end{thm}

\begin{proof} Let $\gamma$ be an oriented path on $(G,\rightarrow)$ of length $n$. To show that $\gamma$ is a geodesic, it suffices to prove that $d(\gamma(0),\gamma(n)) \geq n$. By definition of the $W_1$-orientation, for each $i \in \{0,\ldots n-1\}$ there exists a geodesic $\gamma^{(i)}$ of length $L_i \geq 1$ and $k_i \in \{0,\ldots L_i-1\}$ such that
\begin{itemize}
\item $(e_0(\gamma^{(i)}),e_1(\gamma^{(i)})) \in \mathcal{C}(f_0,f_1)$,
\item $\gamma^{(i)}(k_i)=\gamma(i)$,
\item $\gamma^{(i)}(k_i+1)=\gamma(i+1)$.
\end{itemize}
By Lemma~\ref{lem:suppC}, setting $x_i:=e_0(\gamma^{(i)})$ and $y_i:=e_1(\gamma^{(i)})$, we have
\begin{equation}
\sum_{i=0}^{n-1} d(x_i,y_i) \leq d(x_0,y_{n-1}) + \sum_{i=0}^{n-2} d(x_{i+1},y_i).
\end{equation}
But, $\gamma^{(i)}$ being a geodesic of $G$, we have $d(x_i,y_i)=d(e_0(\gamma^{(i)}),e_1(\gamma^{(i)}))=L_i$. Furthermore, for $i \in \{0,\ldots n-2\}$ we have
\begin{eqnarray*}
d(x_{i+1},y_i) & \leq & d(x_{i+1},\gamma(i+1))+d(\gamma(i+1),y_i) \\
&=& d(\gamma^{(i+1)}(0),\gamma^{(i+1)}(k_{i+1})) + d(\gamma^{(i)}(k_i+1),\gamma^{(i)}(L_i)) \\ 
&=& k_{i+1}+L_{i}-k_i-1.
\end{eqnarray*}
We also have the estimation
\begin{eqnarray*}
d(x_0,y_{n-1}) &\leq & d(x_0,\gamma(0))+d(\gamma(0),\gamma(n))+d(\gamma(n),y_{n-1}) \\
&=& d(\gamma^{(0)}(0),\gamma^{(0)}(k_0))+d(\gamma(0),\gamma(n))\\&&+d(\gamma^{(n-1)}(L_{n-1}),\gamma^{(n-1)}(k_{n-1}+1))\\
&=& k_0+L_{n-1}-k_{n-1}-1+d(\gamma(0),\gamma(n)).
\end{eqnarray*}
We finally have
\begin{equation}
\sum_{i=0}^{n-1} L_i \leq k_0+L_{n-1}-k_{n-1}-1+d(\gamma(0),\gamma(n)) + \sum_{i=0}^{n-2} k_{i+1}+L_{i}-k_i-1,
\end{equation}
which gives $0 \leq d(\gamma(0),\gamma(n)) -n$ and proves the theorem.
\end{proof}

The following shows that the $W_1$-orientation is in some sense stable by restriction:

\begin{prop} \label{prop:W1restriction}
Let $(f_t)_{t \in [0,1]}$ be a $W_1$-geodesic on $G$. For $0 \leq s \leq t \leq 1$, let $(x,y)$ in $E(G)$ such that $x \rightarrow y$ for the $W_1$-orientation with respect to $f_s,f_t$. Then $x \rightarrow y$ for the $W_1$-orientation with respect to $f_0,f_1$.
\end{prop}

\begin{proof} It suffices to show that, if $\tilde{\pi} \in \Pi_1(f_s,f_t)$ and $\tilde{\pi}(b,c) > 0$ then $b \leq c$ for the partial order coming from the $W_1$-orientation w.r.t. $f_0,f_1$. 

\medskip

The proof of this fact is inspired by the 'gluing lemma' stated and explained in~\cite{LottVillani}: let $\pi^{(1)} \in \Pi_1(f_0,f_s), \pi^{(2)} \in \Pi_1(f_s,f_t)$ and $\pi^{(3)} \in \Pi_1(f_t,f_1)$. We consider the 'gluing' $\pi$ of these three couplings, defined by: $$\pi(a,d) := \sum_{b,c \in G} \frac{\pi^{(1)}(a,b) \pi^{(2)}(b,c) \pi^{(3)}(c,d)}{f_s(b) f_t(c)},$$ where the quotient is zero when $f_s(b)=0$ or $f_t(c)=0$. It is easily shown that $\pi \in \Pi(f_0,f_1)$. Moreover,
\begin{eqnarray*}
W_1(f_0,f_1) &\leq & \sum_{a,d \in G} d(a,d) \pi(a,d) \\
& \leq& \sum_{a,b,c,d \in G} (d(a,b)+d(b,c)+d(c,d)) \pi(a,d) \\
& = & \sum_{a,b,c,d \in G} (d(a,b)+d(b,c)+d(c,d)) \frac{\pi^{(1)}(a,b) \pi^{(2)}(b,c) \pi^{(3)}(c,d)}{f_s(b) f_t(c)} \\
&=& \sum_{a,b \in G} d(a,b) \pi^{(1)}(a,b) + \sum_{b,c \in G} d(b,c) \pi^{(2)}(b,c) + \sum_{c,d \in G} d(c,d) \pi^{(3)}(c,d) \\
&=& W_1(f_0,f_s) + W_1(f_s,f_t) + W_1(f_t,f_1) \\
&=& W_1(f_0,f_1).
\end{eqnarray*}
This shows the $W_1$-optimality of $\pi$ and the equality $$d(a,d) \pi(a,d) = (d(a,b)+d(b,c)+d(c,d)) \pi(a,d).$$  Theorem~\ref{prop:GeodG} shows that whenever $\pi(a,d) > 0$, we have $a \leq b \leq c \leq d$. On the other hand, if $\pi^{(2)}(b,c) > 0$ then there exists $a \in \supp(f_0)$ and $d \in \supp(f_1)$ with $\pi^{(1)}(a,b)>0$ and $\pi^{(3)}(c,d) >0 $, so $\pi(a,b)=0$ and so $b \leq c$. 
\end{proof}

\medskip

We now prove:

\begin{thm}\label{th:W1OrientGeod}
Let $(f_t)$ be a smooth $W_1$-geodesic on $G$. We endow this graph with the $W_1$-orientation with respect to $f_0,f_1$. There exists a family $(g_t) : E(G) \rightarrow \mathbb{R}$  such that 
\begin{itemize}
 \item $\forall x \in G$, $\frac{\partial}{\partial t} f_t(x) = -\nabla  g_t(x)$.
 \item $\forall (xy) \in E(G) \ , \ g_t(xy) \geq 0.$
\end{itemize}
Moreover, there exists a family $h_t : T(G) \rightarrow \mathbb{R}$ such that 
\begin{equation*}
\forall (xy) \in E(G) \ , \ \frac{\partial}{\partial t} g_t(xy) = -\nabla  h_t(xy).
\end{equation*}
\end{thm}

We first prove a general result implying the existence of a family $(g_t)$ such that $\frac{\partial}{\partial t} f_t(x) = -\nabla  g_t(x)$:

\begin{lem}\label{lem:unablag}
Let $(G,\rightarrow)$ be an oriented graph and $u : G  \rightarrow \mathbb{R}$ finitely supported such that $ \sum_{x \in G} u(x) = 0$. Then there exists $g : (E(G),\rightarrow) \rightarrow \mathbb{R}$ with $\nabla g=u$.
\end{lem}

\begin{proof} We consider two scalar products, on the spaces of functions defined respectively on $G$ and $E(G)$, defined by $$<u,v>_G := \sum_{x \in G} u(x) v(x) \ , \ <a,b>_E := \sum_{x \rightarrow y} a(xy) b(xy).$$ The adjoint of the divergence operator $\nabla$ is $-\partial$, where $\partial$ is the linear operator $\partial$ defined by $(\partial u) (xy) := u(y)-u(x)$, in the sense that 
\begin{equation} \label{eq:nablapartialadjoint}
<\nabla a,u>_G = -<a,\partial u>_E 
\end{equation}
for any couple $u,a$ of functions respectively defined on $G$ and $E(G)$. The kernel of $ \partial$ is the one-dimensional space generated by the constant function $v=1$. The condition $ \sum_{x \in G} u(x) = 0$ is thus equivalent to $<u,v>_G=0$ or $u \in (\ker(\partial))^{\perp_G}$. We thus want to prove the inclusion $(\ker(\partial))^{\perp_G} \subset \im(\nabla)$. As the linear spaces we are considering are finite-dimensional, this inclusion is equivalent to $(\im(\nabla))^{\perp_G} \subset \ker(\partial)$. Let $u \in (\im(\nabla))^{\perp_G}$. Then for any $b : (E(G),\rightarrow) \rightarrow \mathbb{R}$ we have $<\nabla b, u>_G=0$, so $<b,\partial u>_E = 0$, which proves that $ u \in \ker(\partial)$.
\end{proof}

\medskip

As we have $\sum_{x \in G} \frac{\partial}{\partial t} f_t(x) = 0$, Lemma~\ref{lem:unablag} gives the existence of a family $(g_t)$ with $\frac{\partial}{\partial t} f_t(x) = -\nabla  g_t(x)$. However, this result does not provide an explicit construction of $g$ and in general nothing can be said about its sign.

\medskip

\textit{Proof of Theorem~\ref{th:W1OrientGeod}.} Let $G'$ be a spanning tree of $G$, i.e. a tree having the same vertices as $G$, but with possibly fewer edges. We endow $G'$ with the restriction of the orientation on $G$. According to Lemma~\ref{lem:unablag}, there exists a family of functions $(g_t) : E(G') \rightarrow \mathbb{R}_+$ satisfying $\frac{\partial}{\partial t} f_t(x) = -\nabla  g_t(x)$. As $G'$ is a tree, we know that removing an edge $(x_0y_0)$ from the graph $G'$ will cut it into two disjoint subgraphs $G'_1 := G'_1(x_0y_0)$ and $G'_2 := G'_2(x_0y_0)$ such that $x_0 \in G_1'$ and $y_0 \in G_2'$. Let $u_{(x_0y_0)}$ be the indicator function of $G_1'$. This function satisfies $(\partial u_{(x_0y_0)})(xy) = -1$ if $(xy)=(x_0y_0)$ and  $(\partial u_{(x_0y_0)})(xy) = 0$ otherwise, which implies: 
\begin{eqnarray*}
g_t(x_0y_0) &=&  - \sum_{(xy) \in E(G'} g_t(xy) (\partial u_{(x_0y_0)})(xy) \\
&=& -<g_t, \partial u_{(x_0y_0)}>_E =  <\nabla g_t, u_{(x_0y_0)}>_G \\
&=& -<\frac{\partial}{\partial t} f_t, u_{(x_0y_0)}> = - \sum_{z \in G'_1} \frac{\partial}{\partial t}f_t(z).
\end{eqnarray*}
We want to prove that $g_t(x_0y_0) \geq 0$. Actually we will prove that the function $t \mapsto \sum_{z \in G'_1} f_t(z)$ is strictly decreasing, so we have $g_t(x_0y_0) > 0$. For $0 \leq s \leq t \leq 1$, let $\pi \in \Pi_1(f_s,f_t)$. We have: $$\sum_{z \in G'_1} f_s(z) = \sum_{x \leq y \in G \ : \ x \in G'_1} \pi(x,y) \ , \ \sum_{z \in G'_1} f_t(z) = \sum_{x \leq y \in G \ : \ y \in G'_1} \pi(x,y)$$ By Proposition~\ref{prop:W1restriction}, we know that if $\pi(x,y)>0$ then $x \leq y$. In particular, we cannot have $x \in G_2'$ and $y \in G_1'$. Equivalently, if $x\leq y$, $\pi(x,y)>0$ and $y \in G_1$ then $x \in G_1$. Consequently, we have
\begin{equation} \label{eq:fsdecrease}
\sum_{z \in G'_1} f_s(z)-\sum_{z \in G'_1} f_t(z) = -\sum_{x \leq y  \in G \ : \ x \in G'_1 , y \in G'_2} \pi(x,y) \leq 0.
\end{equation}
Furthermore, as $(x_0y_0)$ is an oriented edge, we know by the definition of $W_1$-orientation that there exists $(x,y) \in \mathcal{C}(f_0,f_1)$ such that $x \leq x_0 \leq y_0 \leq y$. In particular, $x \in G_1'$, $y \in G_2'$ and $\pi(x,y)>0$. This proves that the inequality~\eqref{eq:fsdecrease} is actually strict, which shows the positivity of the family of functions $(g_t)$ on $E(G')$. The first point of Theorem~\ref{th:W1OrientGeod} is proven by extending $g_t$ to $E(G)$, setting $g_t(xy):=0$ if $(xy) \notin E(G')$.

\medskip

The existence of a family of functions $(h_t)$ such that $\frac{\partial }{\partial t} g_t = -\nabla h_t $ is proven by Lemma~\ref{lem:unablag}. We only need to check that $\sum_{(xy) \in E(G)} \frac{\partial }{\partial t} g_t(xy) = 0$. We are actually going to prove the stronger statement: $$\sum_{(xy) \in E(G)} g_t(xy) = W_1(f_0,f_1).$$  To prove this fact, we consider the function $u := \sum_{(x_0y_0) \in E(G')} u_{(x_0y_0)}.$ The function $u$ satisfies $(\partial u)(xy) = 1$ for every $x\rightarrow y \in E(G')$. We then have:
\begin{eqnarray*}
\int_0^t \sum_{(xy) \in E(G)} g_s(xy) ds &=& \int_0^t \sum_{(xy) \in E(G)} g_s(xy) (\partial u)(x,y)ds \\
&=& - \int_0^t \sum_{x \in G} (\nabla g_s)(x) u(x) dt \\
&=& \int_0^t \sum_{x \in G} \frac{\partial}{\partial s} f_s(x) u(x) ds \\
&=& \sum_{y \in G}f_t(y)u(y) - \sum_{x \in G} f_0(x) u(x). 
\end{eqnarray*}
Let $\pi \in \Pi_1(f_0,f_t)$. We know by Proposition~\ref{prop:W1restriction} that if $\pi(x,y)>0$ then $x \leq y$. On the other hand, if $x \leq y$ then there exists a path $x= \gamma_0 \rightarrow \cdots \rightarrow \gamma_n = y$ and we have $u(y)-u(x) = (u(\gamma_n)-u(\gamma_{n-1}))+\cdots +(u(\gamma_1)-u(\gamma_0))= n = d(x,y)$, so we have 
$$\int_0^t \sum_{(xy) \in E(G)} g_s(xy) ds = \sum_{x \leq y} \pi(x,y) d(x,y) = W_1(f_0,f_t) = t W_1(f_0,f_1).$$ Differentiating with respect to $t$ shows that the sum $\sum_{(xy) \in E(G)} g_t(xy)$ is constant and equal to $W_1(f_0,f_1)$. To finish the proof of the theorem, we extend $(h_t)$ to $T(G)$ by defining $h_t(x_0x_1x_2) := 0$ if $(x_0x_1x_2) \notin T(G')$. $\square$ 

\medskip

Actually, Theorem~\ref{th:W1OrientGeod} can be strengthened in the following way:

\begin{prop}\label{prop:W1gpositive}
In Theorem~\ref{th:W1OrientGeod}, we can replace the assertion $\forall (xy) \in E(G) \ , \ g_t(xy) \geq 0$ by $\forall (xy) \in E(G) \ , \ g_t(xy) > 0$.
\end{prop}

\begin{proof}
The proof of Theorem~\ref{th:W1OrientGeod} allowed us to construct, given a spanning tree $G' \subset G$, a family of functions $(g_t^{G'})$ such that $g_t^{G'}(xy)>0$ when $(xy) \in E(G')$ and $g_t^{G'}(xy)=0$ when $(xy) \notin E(G')$. But for each edge $(x_0y_0) \in E(G)$ there exists a spanning tree $G' \subset G$ with $(x_0y_0) \in E(G')$. We define a family $(g_t) : E(G) \rightarrow \mathbb{R}$ as the barycenter
\begin{equation*}
\forall (xy) \in E(G) \ , \ g_t(xy) := \frac{1}{\mathcal{|T|}} \sum_{G' \in  \mathcal{T}} g_t^{G'}(xy),
\end{equation*}
where $\mathcal{T}$ is the (finite) set of spanning trees for $G$. Then $g_t > 0$ and satisfies the conditions of Theorem~\ref{th:W1OrientGeod}. We finally construct a suitable family $(h_t)$ by defining $ h_t := \frac{1}{\mathcal{T}} \sum_{G' \in  \mathcal{|T|}} h_t^{G'}$, where $(h_t^{G'})$ is constructed from $(g_t^{G'})$ as in the proof of Theorem~\ref{th:W1OrientGeod}.
\end{proof}

\subsection{Definition of $W_{1,+}$-geodesics}

Having now constructed an orientation of $G$ associated to each couple of finitely supported probability distributions $f_0,f_1 \in \mathcal{P}(G)$, we propose a definition of $W_{1,+}$-geodesic inspired by Proposition~\ref{prop:BBContraction}:

\begin{defi}\label{def:BBCurves}
Let $G$ be a graph, $W_1$-oriented with respect to a couple of finitely supported probability distributions $f_0,f_1$. A family  $(f_t)$ is said to be a $W_{1,+}$-geodesic if:
\begin{enumerate}
\item The curve $(f_t)$ is a $W_1$-geodesic.
\item There exists two families  $(g_t)$ and $(h_t)$ defined respectively on $E(G)$ and $T(G)$ such that $$\frac{\partial}{\partial t} f_t = - \nabla g_t \ , \ \frac{\partial}{\partial t} g_t = - \nabla h_t.$$
\item For every $(xy) \in E(G)$ we have $g_t(xy) >0$. 
\item The triple $(f_t,g_t,h_t)$ satisfies the Benamou-Brenier equation
\begin{equation} \label{eq:BBgraph}
\forall (x_0x_1x_2) \in T(G) \ , \ f_t(x_1) h_t(x_0x_1x_2) = g_t(x_0x_1)g_t(x_1x_2).
\end{equation}
\end{enumerate}
\end{defi}

\begin{remark} In the sequel,the assertion ``let $(f_t)$ be a $W_{1,+}$-geodesic'' means ``let $((f_t),(g_t),(h_t))$ be a triple of families of functions satisfying the conditions of Definition~\ref{def:BBCurves}''. This is an abuse because nothing is a priori known about the uniqueness of the families $(g_t)$ and $(h_t)$ associated to a $W_{1,+}$-geodesic. 
\end{remark}

\begin{remark} We can check that any contraction of measure on a graph is also a $W_{1,+}$-geodesic: if $f_0=\delta_o$ is a Dirac measure, then the set $\Pi_1(f_0,f_1)$ has only one element, and it easy to prove that the $W_1$-orientation with respect to $f_0,f_1$ coincide with the orientation used for contraction of measures. Proposition~\ref{prop:BBContraction} shows that the other points of Definition~\ref{def:BBCurves} are satisfied by contraction families. 
\end{remark}

It is possible to state \eqref{eq:BBgraph} in terms of two different velocity fields:

\begin{prop} \label{prop:BBCvelocity}
Let $(f_t)_{t \in [0,1]}$ be a $W_{1,+}$-geodesic on $G$.  We define the velocity fields $v_{+,t}$ and $v_{-,t}$ by
\begin{equation} \label{eq:defvfield}
v_{+,t}(x_0x_1) := \frac{g_t(x_0x_1)}{f_t(x_0)} \ , \ v_{-,t}(x_0x_1) := \frac{g_t(x_0x_1)}{f_t(x_1)}
\end{equation}
and the velocity functions $V_{+,t}$ and $V_{-,t}$ by
\begin{equation} \label{eq:defVfunction}
V_{+,t}(x_1) := \sum_{x_2 \in \mathcal{F}(x_1)} v_{+,t}(x_1x_2) \ , \ V_{-,t}(x_1) := \sum_{x_0 \in \mathcal{E}(x_1)} v_{-,t}(x_0x_1).
\end{equation}
The following differential equations then hold:
\begin{equation}
\frac{\partial}{\partial t} v_{+,t}(x_0x_1) = - v_{+,t}(x_0x_1) \left[V_{+,t}(x_1)-V_{+,t}(x_0)\right],
\end{equation}
\begin{equation}
\frac{\partial}{\partial t} v_{-,t}(x_0x_1) = - v_{-,t}(x_0x_1) \left[V_{-,t}(x_1)-V_{-,t}(x_0)\right].
\end{equation}
\end{prop}

\begin{proof} We use the definitions of $g_t$ and $h_t$ and then apply the Benamou-Brenier equation~\eqref{eq:BBgraph} to write:
\begin{eqnarray*}
\frac{\partial}{\partial t} v_{+,t}(x_0x_1) &=& \frac{g_t(x_0x_1)}{f_t(x_0)^2} \left[\sum_{\tilde{x}_1 \in \mathcal{F}(x_0)} g_t(x_0 \tilde{x}_1) - \sum_{x_{-1} \in \mathcal{E}(x_0)} g_t(x_{-1}x_0) \right] \\
&& + \frac{1}{f_t(x_0)} \left[-\sum_{x_2 \in \mathcal{F}(x_1)} \frac{g_t(x_0x_1)g_t(x_1x_2)}{f_t(x_1)}+ \sum_{x_{-1} \in \mathcal{E}(x_0)} \frac{g_t(x_{-1}x_0)g_t(x_0x_1)}{f_t(x_0)} \right] \\
&=& v_{+,t}(x_0x_1) \left[\sum_{\tilde{x}_1 \in \mathcal{F}(x_0)} \frac{g_t(x_0\tilde{x}_1)}{f_t(x_0)} - \sum_{x_2 \in \mathcal{F}(x_1)} \frac{g_t(x_1x_2)}{f_t(x_1)} \right] \\
&=& v_{+,t}(x_0x_1) \left[V_{+,t}(x_0) -V_{+,t}(x_1) \right].
\end{eqnarray*}
The second formula is proven by similar methods.
\end{proof}

\medskip

We now give some heuristic arguments explaining the terminology '$W_{1,+}$-geodesic'. Let us consider the minimization problem described by equation~\eqref{eq:BBproblem} of Theorem~\ref{th:WpBB}, when the paramater $p=1+\ve$ is close to $1$. We use the expansion $a^{1+\ve} = a \exp(\ve \log(a)) = a+ \ve a \log(a) + O(\ve^2)$, valid for $a>0$, to write
$$\int_M \int_0^1 |v_t(x)|^p d\mu_t(x) dt = \int_M \int_0^1 |v_t(x)| d\mu_t(x) dt+ \ve \int_M \int_0^1 |v_t(x)|\log(|v_t(x)|) d\mu_t(x) dt + O(\ve^2).$$
The integral $\int_M \int_0^1 |v_t(x)| d\mu_t(x) dt$ is exactly equation~\eqref{eq:BBproblem} for $p=1$. We thus know, by Theorem~\ref{th:WpBB} that the minimizers of this integral over the set of families $(f_t)$ of probability measures with $f_0,f_1$ prescribed and $\frac{\partial}{\partial t}f_t(x) + \nabla \cdot (v_t(x) f_t(x))=0$ are exaclty the $W_1$-geodesics joining $f_0$ to $f_1$. This suggests the following:

\begin{defi} \label{defi:W1plusCont}
We say that a curve $(f_t)$ of probability measures on a Riemannian manifold $M$ is a $W_{1,+}$-geodesic on $M$ if it is solution to the minimization problem $$\inf \int_M \int_0^1 |v_t(x)|\log(|v_t(x)|) d\mu_t(x) dt,$$ where the infimum is taken over the set of all $W_1$-geodesics between $f_0$ and $f_1$ and where the velocity field $(v_t)$ is defined by the continuity equation $$\frac{\partial}{\partial t} f_t(x) + \nabla \cdot (v_t(x) f_t(x))=0.$$
\end{defi}

The formal optimality condition on $(v_t)$ obtained by applying Euler-Lagrange equations is the same as for $W_p$-geodesics:
$$\frac{\partial}{\partial t} v_t(x) = - v_t(x) \nabla v_t(x).$$

The next proposition shows that $W_{1,+}$-geodesics on a graph can be related to a minimization problem similar to the continuous one described in Definition~\ref{defi:W1plusCont}:

\begin{prop} \label{prop:W1plusProblem}
Let $G$ be a $W_1$-orientated with respect to $f_0,f_1$ finitely supported. We consider the problem
\begin{equation} \label{eq:W1plusBB}
\inf \mathcal{I}_+(f,g) := \inf \int_0^1 \sum_{x\rightarrow y} v_{+,t}(xy) \log(v_{+,t}(xy)) f_t(x),
\end{equation}
where the infimum is taken over the set of $W_1$-geodesics $(f_t)$ between $f_0$ and $f_1$ such that the velocity $v_{+,t}(xy)$ is defined by equation~\eqref{eq:defvfield} from a positive family $(g_t)$ with $\frac{\partial}{\partial t}f_t(x) = -\nabla g_t(x)$.

\medskip

We suppose that there exists a $W_{1,+}$-geodesic $(f_t)$ joining $f_0$ to $f_1$. Then $(f_t)$ is a critical point for $\mathcal{I}_+$ in the following sense: if $(u_t)$ is a family of functions defined on $E(G)$ satisfying the boundary conditions $u_0(xy)=u_1(xy)=0$, then
\begin{equation}
\mathcal{I}_+\left(f+ \eta \nabla u,g- \eta \frac{\partial u}{\partial t}\right) = \mathcal{I}_+(f,g) +O(\eta^2).
\end{equation}
\end{prop}

\textbf{Remark.}Recall that, given a $W_1$-geodesic $(f_t)$, the continuity equation $\frac{\partial}{\partial t}f_t(x) = -\nabla g_t(x)$ may be solved by a family $(g_t)$ which is not necessarily always positive. We restrict ourselves to the families of positive $(g_t)$, which always exist by Proposition~\ref{prop:W1gpositive}, in order to write $|v_{+,t}(xy)|=v_{+,t}(xy)$.

\medskip

\textit{Proof of Proposition~\ref{prop:W1plusProblem}.} When $\eta$ is small, we have the expansion 
$$\mathcal{I}_+\left(f+ \eta \nabla u,g- \eta \frac{\partial u}{\partial t}\right) - \mathcal{I}_+(f,g)$$ $$= - \eta \int_0^1 \sum_{x \rightarrow y} \frac{\partial}{\partial t} u_t(x,y)\left(1+\log(v_{+,t}(xy)) \right)+ (\nabla u_t)(x) v_t(xy) dt + O(\eta^2).$$

On the other hand, we use the boundary conditions $u_0=u_1=0$ to write:
\begin{eqnarray*}
\int_0^1 \sum_{x \rightarrow y} \frac{\partial}{\partial t} u_t(x,y)\left(1+\log(v_{+,t}(xy)) \right) dt &=& - \int_0^1 \sum_{x \rightarrow y}  u_t(x,y) \frac{\partial}{\partial t} \left(1+\log(v_{+,t}(xy)) \right) dt\\
&=& -\int_0^1 \sum_{x \rightarrow y} u_t(xy) \frac{1}{v_t(xy)} \frac{\partial}{\partial t}v_t(xy) dt\\
&=&  \int_0^1 \sum_{x \rightarrow y} u_t(xy) [V_{+,t}(y)-V_{+,t}(x)] dt\\
&=& -\int_0^1 \sum_{x \in G} (\nabla u_t)(x) V_{+,t}(x) dt\\
&=& \int_0^1 \sum_{x \rightarrow y} (\nabla u_t)(x) v_t(xy) dt, 
\end{eqnarray*}

which proves that $\mathcal{I}_+\left(f+ \eta \nabla u,g- \eta \frac{\partial u}{\partial t}\right) = \mathcal{I}_+(f,g) +O(\eta^2).$ $\square$

\medskip

\textbf{Remark. }Similarly, it can be proven that a $W_{1,+}$-geodesic is also critical for the functional \begin{equation}
\inf \mathcal{I}_-(f,g) := \inf \int_0^1 \sum_{x\rightarrow y} v_{-,t}(xy) \log(v_{-,t}(xy)) f_t(x).
\end{equation}

\section{$W_{1,+}$-geodesics as mixtures of binomial distributions} \label{sec:Structure}

$W_{1,+}$-geodesics have been constructed as generalizations of contraction families, which have been defined as mixture of binomial distributions. In this section, we fix a $W_{1,+}$-geodesic $(f_t)$ on $G$, joining two finitely supported probability distributions $f_0,f_1 \in \mathcal{P}(G)$. It will always be assumed that the graph $G$ is $W_1$-oriented with respect to $f_0,f_1$ and that every path is an oriented path, thus a geodesic, by Theorem~\ref{prop:GeodG}. 

\medskip

The main purpose of this section is to prove Theorem~\ref{thm:BBCBinom}: $(f_t)$ can also be expressed as a mixture of binomial measures, with respect to a coupling $\pi \in \Pi(f_0,f_1)$ solution to a certain minimization problem. The key ingredients to the proof of this theorem are the study of the behaviour of $(f_t)$ along particular geodesics of $G$, called extremal and semi-extremal geodesics, and the construction of two sub-Markov kernels $K,K^\star$ on $G$ associated to $(f_t)$.

\subsection{Extremal geodesics}

Recall that we write $x_2 \in \mathcal{F}(x_1)$ and $x_1 \in \mathcal{E}(x_2)$ if $x_1 \leq x_2$ and $d(x_1,x_2)=1$ or equivalently if $(x_1x_2)$ is an oriented edge of $G$. If $\gamma$ is a geodesic of $G$, it will be sometimes convenient to use the notation $\gamma_i := \gamma(i)$. 

\begin{defi}
Let $\gamma$ be a geodesic on $G$. 
\begin{itemize}
\item If $L(\gamma) = n \geq 2$, we associate to $\gamma$ the positive function
\begin{equation}
C_\gamma(t) := \frac{g_t(\gamma_0\gamma_1) \cdots g_t(\gamma_{n-1}\gamma_n)}{f_t(\gamma_1) \cdots f_t(\gamma_{n-1})}.
\end{equation}
\item If $L(\gamma)=1$, we define $C_\gamma(t) := g_t(\gamma_0\gamma_1)$.
\item If $L(\gamma)=0$, we define $C_\gamma(t) := f_t(\gamma_0)$.
\end{itemize}
\end{defi}

\begin{prop}\label{prop:Cgammadiff}
The function $C_\gamma(t)$ satisfy
\begin{equation} \label{eq:Cgammadiff}
\frac{\partial}{\partial t} C_\gamma(t) =  \sum_{x_0 \in \mathcal{E}(\gamma_0)} C_{x_0 \cup \gamma}(t) -\sum_{x_2 \in \mathcal{F}(\gamma_n)} C_{\gamma \cup x_2}(t),
\end{equation}
where $x_0 \cup \gamma$ (resp. $\gamma_n \cup x_2$) is the geodesic $x_0,\gamma_0,\ldots \gamma_n$ (resp. $\gamma_0,\ldots \gamma_n,x_2$).
\end{prop}

\begin{proof} If $L(\gamma)=0$, equation~\eqref{eq:Cgammadiff} is equivalent to $\frac{\partial}{\partial t}f_t(\gamma_0) = -(\nabla g_t)(x_0)$, which is true by the definition of $(g_t)$. If $L(\gamma) \geq 1$, we notice that 
\begin{equation}\label{eq:Cgammavt}
C_\gamma(t) = f_t(\gamma_0) v_{+,t}(\gamma_0 \gamma_1) \cdots v_{+,t}(\gamma_{n-1}\gamma_n).
\end{equation}
Proposition~\ref{prop:BBCvelocity} gives:
\begin{eqnarray*}
\frac{1}{C_\gamma(t)} \frac{\partial}{\partial t} C_\gamma(t) &=& \frac{1}{f_t(\gamma_0)}\frac{\partial}{\partial t} f_t(\gamma_0) + \sum_{i=0}^{n-1} \frac{1}{v_{+,t}(\gamma_i \gamma_{i+1})} \frac{\partial}{\partial t} v_{+,t}(\gamma_i \gamma_{i+1}) \\
&=& \left(- V_{+,t}(\gamma_0)+V_{-,t}(\gamma_0)  \right) - \sum_{i=0}^{n-1} \left[V_{+,t}(\gamma_{i+1}) - V_{+,t}(\gamma_{i})\right] \\
&=& -V_{+,t}(\gamma_n)+V_{-,t}(\gamma_0) \\
&=& - \sum_{x_2 \in \mathcal{F}(\gamma_n)} v_{+,t}(\gamma_n x_2)+\sum_{x_0 \in \mathcal{E}(\gamma_0)} v_{-,t}(x_0 \gamma).
\end{eqnarray*}
Multiplying by $C_\gamma(t)$ and applying equation~\eqref{eq:Cgammavt} leads to the result.
\end{proof}

\medskip

Equation~\eqref{eq:Cgammadiff} takes a simpler form in the case where the set $\mathcal{E}(e_0(\gamma))$ (or $\mathcal{F}(e_1(\gamma))$, or both) is empty. This motivates the following:

\begin{defi} 
We define the particular subsets of vertices of $G$:
\begin{itemize}
 \item The set of initial vertices $\mathcal{A} \subset G$ contains every $x_1 \in G$ such that $\mathcal{E}(x_1)$ is empty.
 \item The set of final vertices $\mathcal{B} \subset G$ contains every $x_1 \in G$ such that $\mathcal{F}(x_1)$ is empty.
\end{itemize}
We also define the particular subsets of geodesics f $G$:
\begin{itemize}
\item The set $\EG$ of extremal geodesics contains every $\gamma \in \Gamma(G)$ with $e_0(\gamma) \in \mathcal{A}$, $e_1(\gamma) \in \mathcal{B}$.
\item The set $\SEG_{1,x}$ contains every $\gamma \in \Gamma(G)$ with $e_0(\gamma) \in \mathcal{A}$, $e_1(\gamma) =x $.
\item The set $\SEG_{1,x}$ contains every $\gamma \in \Gamma(G)$ with $e_0(\gamma)=x$, $e_1(\gamma) \in \mathcal{B}$.
\end{itemize}
If $e_0(\gamma) \in \mathcal{A}$ or $e_1(\gamma) \in \mathcal{B}$, the geodesic $\gamma$ is said to be semi-extremal.
\end{defi}

\begin{remark} The sets $\mathcal{A}$ and $\mathcal{B}$ are both non empty. If we suppose for instance that $\mathcal{B}$ is empty, then we can construct an infinite sequence $(x_n)_{n \geq 0}$ in $G$ such that $x_{n+1} \in \mathcal{F}(x_n)$. But, $f_0$ and $f_1$ being finitely supported and $G$ being locally finite, the set of oriented edges of $G$ is finite so $x_p=x_q$ for a couple of indices $q > p$. This means that there exists a non-trivial oriented path $\gamma$ joining $x_p$ to itself, which is a contradiction because $\gamma$ is a geodesic of $G$ by Proposition~\ref{prop:GeodG}. \end{remark}

\medskip

An immediate corollary of Proposition~\ref{prop:Cgammadiff} is the following:

\begin{prop}\label{prop:CgammaConst}
Let $\gamma$ be a geodesic of $G$.
\begin{itemize}
\item If $\gamma \in \EG$, then $C_\gamma(t)=C_\gamma$ is a constant function of $t$.
\item If $\gamma \in\SEG_{1,x}$ then $C_\gamma(t)$ is polynomial in $t$ and $  \degre (C_\gamma(t)) \leq \sup \{L(\tilde{\gamma}) : \tilde{\gamma} \in SE\Gamma_{2,x} \}.$
\item If $\gamma \in \SEG_{2,x}$ then $C_\gamma(t)$ is polynomial in $t$ and $ \degre (C_\gamma(t)) \leq \sup \{L(\tilde{\gamma}) : \tilde{\gamma} \in SE\Gamma_{1,x} \}.$
\end{itemize}
\end{prop}

\begin{proof} If $\gamma \in \EG$, then the sets $\mathcal{E}(e_0(\gamma))$ and $\mathcal{F}(e_1(\gamma))$ are empty, which by Proposition~\ref{prop:Cgammadiff} shows that $C_\gamma$ is a constant function of $t$.
We prove the second point by induction on $m=m(\gamma) := \sup \{L(\tilde{\gamma}) : \tilde{\gamma} \in \SEG_{2,x}\}$, which only depends on the endpoint $e_1(\gamma)=x$ . If $m=0$ then $\gamma \in \EG$ and this case has been considered in the first point. We now fix a geodesic $\gamma \in \SEG_{1,x}$ such that $m(\gamma) \geq 1$. We apply Proposition~\ref{prop:Cgammadiff} and use the fact that $e_0(\gamma) \in \mathcal{A}$ to write:
\begin{equation}\label{eq:SEGdiff}
\frac{\partial}{\partial t} C_\gamma(t)= -\sum_{x_2 \in \mathcal{F}(x)} C_{\gamma\cup x_2}(t).
\end{equation}
It is easily shown that, for $z \in \mathcal{F}(x)$, $m(\gamma \cup \{z\})= m(\gamma)-1$, which proves by induction on $m$ that $C_\gamma(t)$ is polynomial in $t$ of degree less than $m(\gamma)$. 
\end{proof}

\subsection{Sub-Markov kernels associated to a $W_{1,+}$-geodesic}

The fact that the function $C_\gamma$ is constant and positive on extremal geodesics allows us to introduce a useful function on ordered subsets of $G$: 

\begin{defi} \label{def:mfromC}
Given an ordered $p$-uple $z_1 \leq z_2 \leq \ldots z_p$ of vertices of $G$, we define
\begin{equation}
m(z_1,\ldots z_p) := \sum_{\gamma \in E(z_1,\ldots z_p)} C_\gamma,
\end{equation}
where $E(z_1,\ldots z_p) \subset \EG$ is defined by:
\begin{equation}
\gamma \in E(z_1,\ldots z_p) \Longleftrightarrow \exists \ k_1 \leq \cdots \leq k_p \ , \ \gamma(k_i)=z_i.
\end{equation}
If $\gamma$ is a geodesic of $G$, we denote by $m(\gamma)$ the number $m(e_0(\gamma(0)),\cdots e_1(\gamma))$.
\end{defi}

\begin{prop}\label{prop:Exbiject}
For any family of vertices $x_0 \leq \cdots \leq x_m$ we have
\begin{equation*}
m(x_0,\ldots x_m) = \frac{m(x_0,x_1) \cdots m(x_{m-1,}x_m)}{m(x_1) \cdots m(x_{m-1})}. 
\end{equation*}
\end{prop}

\begin{proof} Let $\gamma \in E(x_0,\cdots x_m)$ and $\gamma^{(i)} \in E(x_i)$ for $i=1,\ldots m-1$. To these geodesics we associate the geodesics $\tilde{\gamma}^{(0)} \cdots \tilde{\gamma}^{(m-1)}$ such that, for $i=1,\ldots m-1$, $\tilde{\gamma}^{(i)}$ is constructed by concatenating the begining of $\gamma^{(i)}$, a mid-part of $\gamma$ and the end of $\gamma^{(i+1)}$ in the following way:
\begin{equation}
\tilde{\gamma}^{(i)} : e_0(\gamma^{(i)}) \rightarrow \cdots x_i \rightarrow  \cdots x_{i+1} \rightarrow \cdots e_1(\gamma^{(i+1)}).
\end{equation}
The geodesic $\tilde{\gamma}^{(0)}$ is constructed by concatenating the begining of $\gamma$ and the end of $\gamma^{(1)}$ and $\tilde{\gamma}^{(m-1)}$ is contructed by concatenating the begining of $\gamma^{(m-1)}$ and the end of $\gamma$. It is clear that $\gamma^{(i)} \in E(x_i,x_{i+1})$. Moreover, the application $(\gamma,\gamma^{(1)},\cdots \gamma^{(m-1)}) \mapsto (\tilde{\gamma}^{(0)},\cdots \tilde{\gamma}^{(m-1)})$ is easily proven to be a bijection between $E(x_0,\ldots x_m) \times E(x_1) \times \cdots E(x_{m-1})$ and $E(x_0,x_1) \times \cdots E(x_{m-1},x_m)$. Writing that both sets have same cardinality gives the result. 
\end{proof}

\begin{defi}
The sub-Markov kernels $K$ and $K^*$ associated to a $W_{1,+}$-geodesic $(f_t)$ on $G$ are defined by
\begin{equation} \label{eq:MarkovKdef}
\forall x_1 \in G \ , \forall x_0 \in \mathcal{E}(x_1) \ , \ K(x_1,x_0) := \frac{m(x_1,x_0)}{m(x_1)},
\end{equation}
\begin{equation} \label{eq:MarkovKstardef}
\forall x_0 \in G \ , \forall x_1 \in \mathcal{F}(x_0) \ , \ K^*(x_0,x_1) := \frac{m(x_0,x_1)}{m(x_0)}.
\end{equation}
We also define $Kf(x_1) := \sum_{x_0} K(x_1,x_0) f(x_0)$ and $K^*f(x_0) := \sum_{x_1} K^*(x_0,x_1) f(x_1)$.
\end{defi}

\begin{prop} \label{prop:KKstarProp}
The kernels $K$ and $K^*$ satisfy the following:
\begin{itemize}
\item If $x_1 \notin \mathcal{A}$ then $\sum_{x_0 \in \mathcal{E}(x_1)} K(x_1,x_0) =1$.
\item If $x_0 \notin \mathcal{B}$ then $\sum_{x_1 \in \mathcal{F}(x_0)} K^*(x_0,x_1) = 1$.
\item The operators $K$ and $K^*$ are adjoint for the scalar product $<f,g> := \sum_{x \in G} f(x) g(x) m(x)$.
\item The iterated kernel $K^n$ is supported on the set of couples $(x_n,x_{0})$ such that $x_0 \in \mathcal{E}^n(x_n)$, i.e. such that $x_0 \leq x_n$ and $d(x_n,x_0)=n$. For such a couple we have $$K^n(x_n,x_0) = \frac{m(x_n,x_0)}{m(x_n)}.$$
\item Similarly, for $x_n \in \mathcal{F}^n(x_0)$, i.e. for $x_{n} \leq x_0$ such that $d(x_{n},x_0)=n$ we have $$(K^*)^n(x_0,x_{n}) = \frac{m(x_0,x_n)}{m(x_0)}.$$
\item The operators $K$ and $K^*$ are nilpotent.
\end{itemize}
\end{prop}

\begin{proof} The first point comes from the fact that, if $x_0 \notin \mathcal{B}$, there exists a bijection between the set $E(x_0)$ and the disjoint union $\bigcup_{x_1 \in \mathcal{F}(x_0)} E(x_0,x_1)$. The second point is proven similarly. The third point is proven by noticing that both scalar products $<Kf,g>$ and $<f,K^*g>$ are equal to $$\sum_{x_0 \rightarrow x_1} m(x_0,x_1) f(x_0) g(x_1).$$ To prove the fourth point, we write the general formula for the iterated kernel for some $n\geq 2$: $$\forall x_0, x_n \in G \ , \ K^n(x_n,x_0) := \sum_{x_{n-1},\ldots x_{1}} K(x_n,x_{n-1})\cdots K(x_{1},x_0).$$ The product $K(x_n,x_{n-1})\cdots K(x_1,x_0)$ is non-zero if and only if $x_0 \rightarrow \cdots \rightarrow x_n$, i.e. if $(x_0,\ldots,x_n)$ is a geodesic. This proves that $K^n(x_n,x_0) > 0$ implies that $x_0 \in \mathcal{E}^n(x_n)$. Moreover we have:
\begin{equation*}
K^n(x_n,x_0) = \sum_{\gamma \in \Gamma_{x_0,x_n}} \frac{m(x_n,\gamma_{n-1})}{m(x_n)} \cdots \frac{m(\gamma_1,x_0)}{m(\gamma_1)} 
= \frac{m(x_n,x_0)}{m(x_n)}
\end{equation*}
by Proposition~\ref{prop:Exbiject}. The fifth point is proven similarly. The nilpotency of $K$ and $K^*$ comes from the fact that $(G,\rightarrow)$ has a finite diameter: if $n > \diam(G)$ then $K^n=0$ and $(K^{*})^n=0$. 
\end{proof}

\medskip

\begin{remark} The first point of Proposition~\ref{prop:KKstarProp} shows that $K$ can easily be transformed into a Markov kernel: it suffices to add a vertex $\omega$ (often called ``cemetery'') to $G$ and oriented edges $\omega \rightarrow x$ for every $x \in \mathcal{A}$. The sub-Markov kernel $K$ is extended into a Markov kernel on $G\cup \omega$ by defining $K(\omega,\omega)=1$ and $K(\omega,x)=1$ for every $x \in \mathcal{A}$. The kernel $K^*$ can be treated similarly, by considering the oriented edges $(x,\omega)$ for $x \in \mathcal{B}.$
\end{remark}

\subsection{Polynomial structure of $W_{1,+}$-geodesics}

In this paragraph we use properties of the functions $C_\gamma(t)$ and of the sub-Markovian kernels $K,K^\star$ to give expression of $(f_t)$ as a mixture of binomial distributions on geodesics of $G$. 

\medskip

A direct calculation proves the following fundamental result:

\begin{prop}\label{prop:Ctft}
Let $x\in G$ be a vertex and $\gamma,\tilde{\gamma}$ be two geodesics on $G$ with $\gamma \in \SEG_{1,x}$ and $\tilde{\gamma} \in \SEG_{2,x}$. Then \begin{equation} f_t(x) = \frac{C_{\gamma}(t) C_{\tilde{\gamma}}(t)}{C_{\gamma\cup \tilde{\gamma}}}, \end{equation} where $\gamma \cup \tilde{\gamma}$ is the concatenation of $\gamma$ and $\tilde{\gamma}$.
\end{prop}

\begin{remark} A first consequence of Propositions~\ref{prop:CgammaConst} and~\ref{prop:Ctft} is the fact that, for any $x \in G$, $f_t(x)$ is a polynomial function of $t$ such that $\degre(f_t(x)) \leq \diam(G)$. \end{remark}

We also use Proposition~\ref{prop:Ctft} to show the following:

\begin{prop}\label{prop:Cgamma1frac}
For $x \in G$, we consider two semi-extremal curves $\gamma^{(1)},\gamma^{(2)} \in \SEG_{1,x}$. The quotient $\frac{C_{\gamma^{(1)}}(t)}{C_{\gamma^{(2)}}(t)}$ does not depend on $t$ and is equal to $\frac{m(\gamma^{(1)})}{m(\gamma^{(2)})}$. Furthemore, we have
\begin{equation}
C_{\gamma^{(1)}}(t) = \frac{m(\gamma^{(1)})}{m(x)} \sum_{\gamma \in \SEG_{1,x}(G)} C_\gamma(t).
\end{equation}
\end{prop}

\begin{proof} Let $\tilde{\gamma}$ be in $\SEG_{2,x}(G)$. Then Proposition~\eqref{prop:Ctft} shows that $$\frac{C_{\gamma^{(1)}}(t)}{C_{\gamma^{(2)}}(t)} = \frac{C(\gamma^{(1)} \cup \tilde{\gamma})}{C(\gamma^{(2)} \cup \tilde{\gamma})}.$$ We use the fact that this quotient does not depend on $\tilde{\gamma}$ to write $$\frac{C_{\gamma^{(1)}}(t)}{C_{\gamma^{(2)}}(t)} = \frac{\sum_{\tilde{\gamma} \in \SEG_{2,x}(G) }C(\gamma^{(1)} \cup \tilde{\gamma})}{\sum_{\tilde{\gamma} \in \SEG_{2,x}(G) }C(\gamma^{(2)} \cup \tilde{\gamma})} = \frac{m(\gamma^{(1)})}{m(\gamma^{(2)})}.$$
The second point is proven by writing
\begin{equation}
\sum_{\gamma \in SE\Gamma_{1,x}(G)} \frac{C_\gamma(t)}{C_{\gamma^{(1)}}(t)} = \sum_{\gamma \in \SEG_{1,x}(G)} \frac{m(\gamma)}{m(\gamma^{(1)})} = \frac{m(x)}{m(\gamma^{(1)})}.
\end{equation}
\end{proof}

We now introduce two families of functions which play the same role as in the case of contraction of measures:

\begin{defi}
We define the functions $P_t(x)$ and $Q_t(x)$ by
\begin{equation}
P_t(x) := \frac{1}{m(x)} \sum_{\gamma^{(2)} \in \SEG_{2,x}} C_t(\gamma^{(2)}) \ , \ Q_t(x) := \frac{1}{m(x)} \sum_{\gamma^{(1)} \in \SEG_{1,x}} C_t(\gamma^{(1)}).
\end{equation}
\end{defi}

\begin{prop}
The functions $f_t$, $g_t$ and $h_t$ are related to $P_t$, $Q_t$ and $m$ by
\begin{enumerate}
\item $f_t(x_0) = m(x_0) P_t(x_0) Q_t(x_0)$,
\item $g_t(x_0x_1) = m(x_0x_1) P_t(x_0)Q_t(x_1)$,
\item $h_t(x_0x_1x_2)= m(x_0x_1x_2) P_t(x_0)Q_t(x_2)$.
\end{enumerate}
\end{prop}

\begin{proof} To prove the first point, we notice that the concatenation map $\gamma^{(1)}, \gamma^{(2)} \mapsto \gamma^{(1)} \cup \gamma^{(2)}$ is a bijection between the sets $\SEG_{1,x_0} \times \SEG_{2,x_0}$ and $E(x_0)$. We then use Proposition~\ref{prop:Ctft} to write:
\begin{eqnarray*}
\sum_{(\gamma^{(1)},\gamma^{(2)}) \in \SEG_{1,x_0} \times \SEG_{2,x_0}} C_{\gamma^{(1)}}(t) C_{\gamma^{(2)}}(t) &=& \sum_{\gamma \in E(x_0)} C_\gamma f_t(x_0) \\
&=& f_t(x_0) m(x_0).
\end{eqnarray*}
To prove the second point, given of vertices $x_0 \rightarrow x_1$ we consider the bijection between the sets $\SEG_{1,x_0} \times \SEG_{2,x_1}$ and $E(x_0,x_1)$ given by the concatenation $\gamma^{(1)},\gamma^{(2)} \rightarrow \gamma^{(1)} \cup \gamma^{(2)}$. Moreover, if $\gamma^{(1)} \in \SEG_{1,x_0}$ and $\gamma^{(2)} \in \SEG_{2,x_1}$ have length $L_1 \geq 2$ and $L_2 \geq 2$ we have:
\begin{eqnarray*}
C_{\gamma^{(1)}}(t) C_{\gamma^{(2)}}(t) &=& \frac{g(\gamma^{(1)}_0 \gamma^{(1)}_1) \cdots g(\gamma^{(1)}_{L_1-1} x_0) }{f(\gamma^{(1)}_1) \cdots f(\gamma^{(1)}_{L_1-1})} \frac{g(x_1 \gamma^{(2)}_1) \cdots g(\gamma^{(2)}_{L_2-1} \gamma^{(2)}_{L_2}) }{f(\gamma^{(2)}_1) \cdots f(\gamma^{(2)}_{L_2-1})} \\
&=& C_{\gamma^{(1)} \cup \gamma^{(2)}}(t) \frac{f_t(x_0) f_t(x_1)}{g_t(x_0x_1)}.
\end{eqnarray*}
Summing over all $\gamma^{(1)},\gamma^{(2)}$ gives $$\frac{1}{m(x_0)} P_t(x_1) \frac{1}{m(x_1)} Q_t(x_0) = m(x_0,x_1) \frac{f_t(x_0) f_t(x_1)}{g_t(x_0x_1)}.$$ Replacing $f_t(x_0)$ and $f_t(x_1)$ by their expressions in terms of $P_t,Q_t$ proves the second point. The third point is simply proven by using the Benamou-Brenier equation:
\begin{eqnarray*}
h_t(x_0x_1x_2) &=& \frac{g_t(x_0x_1)g_t(x_1x_2)}{f_t(x_1)} = \frac{m(x_0,x_1)m(x_1,x_2)}{m(x_1)}P_t(x_0)Q_t(x_2) \\&=& m(x_0,x_1,x_2) P_t(x_0) Q_t(x_2).
\end{eqnarray*}
\end{proof}

\begin{prop} \label{prop:PQPDE}
The functions $P_t$ and $Q_t$ satisfy the differential equations
\begin{equation}
 \frac{\partial}{\partial t}P_t(x) =  KP_t(x)\ , \ \frac{\partial}{\partial t}Q_t(x) = -K^*Q_t(x).
\end{equation}
\end{prop}

\textbf{Proof: } When applied to semi-extremal geodesics, Proposition~\ref{prop:Cgammadiff} takes a simpler form. More precisely, if $\gamma^{(2)} \in SE\Gamma_{2,x_0}$, we have
\begin{equation}
\frac{\partial}{\partial t} C_{\gamma^{(2)}}(t) = \sum_{x_{-1} \in \mathcal{E}(x_0)} C_{x_{-1} \cup \gamma^{(2)}}(t).
\end{equation}
On the other hand, by Proposition~\ref{prop:Cgamma1frac}, we have:
\begin{eqnarray*}
 \sum_{\gamma^{(2)} \in SE\Gamma_{2,x_0}} C_t(x_{-1} \cup \gamma^{(2)}) &=& \sum_{\gamma_2 \in SE\Gamma_{2,x_0}} \frac{m(x_{-1} \cup \gamma^{(2)})}{m(x_{-1})} \sum_{\tilde{\gamma}^{(2)} \in SE\Gamma_{2,x_{-1}}} C_{\tilde{\gamma}^{(2)}}(t)\\
 &=& \frac{m(x_{-1},x_0)}{m(x_{-1})} P_t(x_{-1}).
\end{eqnarray*}
Summing this last equation over $x_{-1} \in \mathcal{E}(x_0)$ gives the result. The differential equation for $Q_t(x_0)$ is proven similarly. $\square$

\begin{prop} \label{prop:abexist}
There exist two functions $a,b : G \rightarrow \mathbb{R}$ such that
\begin{equation}
P_t(z) = \frac{1}{m(z)} \sum_{x \leq z} m(x,z) a(x) \frac{t^{d(x,z)}}{d(x,z)!} \ , \ Q_t(z) = \frac{1}{m(z)} \sum_{y \geq z} m(z,y) b(y) \frac{(1-t)^{d(z,y)}}{d(z,y)!}.
\end{equation}
\end{prop}

\begin{proof} For $x \in G$, let $a(x):=P_0(x)$ be the constant term of the polynomial $t \mapsto P_t(x)$. Using Proposition~\ref{prop:PQPDE} and Proposition~\ref{prop:KKstarProp}, we have 
\begin{eqnarray*}
P_t(z) &=& [\exp(tK)P_0](x) = [\exp(tK)a](z)\\
&=& \sum_{l \geq 0} \frac{t^l}{l!} (K^l a)(z) \\
&=& \sum_{l \geq 0} \sum_{x \in \mathcal{E}^l(z)} \frac{t^l}{l!} \frac{m(x,z)}{m(z)} a(x) \\
&=& \frac{1}{m(z)} \sum_{x \leq z} m(x,z) a(x) \frac{t^{d(x,z)}}{d(x,z)!}.
\end{eqnarray*}
The proof of the second point is quite similar: define $\tilde{Q}_t(z):= Q_{1-t}(z)$ and $b(y) := \tilde{Q}_0(y) = Q_1(y)$. As we have $\frac{\partial \tilde{Q}_t(z)}{\partial t} = (K^* \tilde{Q}_t)(z)$, we use again Proposition~\ref{prop:KKstarProp} to conlude. \end{proof}

\medskip

We are now ready to write the $W_{1,+}$-geodesic $(f_t)$ as a mixture of binomial distributions:

\begin{thm} \label{thm:BBCBinom}
For any couple of vertices $x\leq y \in G$ we define the binomial probability distribution on $\bino_{(x,y),t}$ on $G$, associated to the application $m$, supported on the set of vertices $z \in G$ such that $x \leq z \leq y$, by
\begin{equation} \label{eq:BinoMixture}
\bino_{(x,y),t}(z) := \frac{m(x,z,y)}{m(x,y)} \frac{d(x,y)!}{d(x,z)!d(z,y)!} t^{d(x,z)} (1-t)^{d(x,y)}.
\end{equation}
The $W_{1,+}$-geodesic $(f_t)_{t \in [0,1]}$ is a mixture of such binomial distributions:
\begin{equation} \label{eq:ftgeneralform}
 f_t(\cdot) = \sum_{x \leq y} \frac{m(x,y)}{d(x,y)!} a(x) b(y) \bino_{(x,y),t}(\cdot).
\end{equation}
\end{thm}

\begin{proof} The theorem follows from the calculation:
\begin{eqnarray*}
f_t(z) &=& m(z) P_t(z)Q_t(z) \\
&=& \frac{1}{m(z)} \sum_{x,y : x\leq z \leq y} m(x,z) a(x) \frac{t^{d(x,z)}}{d(x,z)!} m(z,y) b(y) \frac{(1-t)^{d(z,y)}}{d(z,y)!} \\
&=& \sum_{x,y : x\leq z \leq y} \frac{m(x,z)m(z,y)}{m(z)} \frac{m(x,y)}{d(x,y)!} a(x) b(y) \bino_{(x,y),t}(z),
\end{eqnarray*}
and from the fact that $\frac{m(x,z)m(z,y)}{m(z)} = m(x,z,y)$ (by Proposition~\ref{prop:Exbiject}).
\end{proof}

\section{Existence of $W_{1,+}$-geodesics}\label{sec:Existence}

In the previous section, we showed that any $W_{1,+}$-geodesic $(f_t)$ can be expressed a mixture of binomial distributions with respect to a certain coupling between $f_0$ and $f_1$. We now turn to the question of the existence of a $W_{1,+}$-geodesic $(f_t)$ joining two fixed probability distributions $f_0$, $f_1$. Through this section, we fix such a couple and endow the underlying graph $G$ with the $W_1$-orientation associated to $f_0,f_1$.

\begin{defi}\label{def:mextended}
Let $m : E(G) \rightarrow \mathbb{R}_+^*$ be satisfying $\nabla m(x) = 0$ for every $x \notin \mathcal{A},\mathcal{B}$. Let $p  \geq 0$ be an integer. We extend $m$ as a function on ordered $(p+1)-uples$ in $G$ by defining:
\begin{itemize}
\item If $p=0$, $m(x) := \sum_{y \in \mathcal{F}(x)} m(x,y)$.
\item If $p \geq 2$ and $\gamma : \{0,\ldots p\} \rightarrow G$ is a geodesic, then 
\begin{equation} \label{eq:mextendedgeodesics}
m(\gamma) := m(\gamma_0,\ldots \gamma_p) := \frac{\prod_{i=0}^{p-1} m(x_i,x_{i+1})}{\prod_{j=1}^{p-1} m(x_j)}
\end{equation}
\item If $p \geq 2$ and $x_0 \leq \cdots \leq x_p$ then  $$m(x_0,\ldots x_p) = \frac{1}{\prod_{j=1}^{p-1} m(x_j)} \prod_{i=1}^{p-1} \sum_{\gamma \in \Gamma_{x_i,x_{i+1}}} m(\gamma) .$$
\end{itemize}
\end{defi}

\begin{remark} The assumption $\nabla m(x)=0$ for $x \notin \mathcal{A},\mathcal{B}$ allows us to write $$\sum_{y \in \mathcal{F}(x)} m(x,y) = m(x) = \sum_{y' \in \mathcal{E}(x)} m(y',x).$$ \end{remark}

\begin{remark} An equivalent way to define the extension of $m$ is to define $m(\gamma)$ on extremal geodesics using equation~\eqref{eq:mextendedgeodesics} and to extend it to general $(p+1)$-uples as in Definition~\ref{def:mfromC}, the quantity $m(\gamma)$ playing the role of $C_\gamma$. \end{remark}

\begin{thm} \label{thm:BinomBBC}
A $W_1$-geodesic $(f_t)$ is a $W_{1,+}$-geodesic if and only if there exists:
\begin{itemize} 
\item A function $m : E(G) \rightarrow \mathbb{R}_+^*$ satisfying $\nabla m(x) = 0$ for $x \notin \mathcal{A},\mathcal{B}$, extended to ordered families of $G$.
\item A couple of non-negative functions $a,b : G \rightarrow \mathbb{R}_+$, 
\end{itemize}
such that equations~\eqref{eq:BinoMixture} and \eqref{eq:ftgeneralform} hold.
\end{thm}

\begin{proof} The ``only if'' part of Theorem~\ref{thm:BinomBBC} is exactly Theorem~\ref{thm:BBCBinom}. Indeed, the restriction to $E(G)$ of the function $m$ constructed from a $W_{1,+}$-geodesic $(f_t)$ satisfies $\nabla m=0$ outside of $\mathcal{A} \cup \mathcal{B}$, and using Definition~\ref{def:mextended} to extend this restriction to ordered families allows us to recover the original $m$. Moreover, the functions $a$ and $b$ introduced in Proposition~\ref{prop:abexist} are non-negative: $a(x)$ is the constant term of the polynomial $P_t(x)$, which is non-negative for every $t \in [0,1]$, and the same goes for $b(x)$.

\medskip

Conversely, let $(f_t)$ be a curve satisfying the assumptions of Theorem~\ref{thm:BinomBBC}. We define the polynomial functions $$P_t(z) := \frac{1}{m(z)} \sum_{x \leq z} m(x,z) a(x) \frac{t^{d(x,z)}}{d(x,z)!} \ , \ Q_t(z) := \frac{1}{m(z)} \sum_{y \geq z} m(z,y) b(y) \frac{t^{d(z,y)}}{d(z,y)!}.$$ Direct calculations show that $f_t(z) = m(z) P_t(z) Q_{1-t}(z)$. Moreover, using the definition of $m(x,z)$ and $m(z,y)$, one can prove easily that $P_t$ and $Q_t$ satisfy the differential equations
$$\frac{\partial }{\partial t} P_t(z) = \sum_{z_0 \in \mathcal{E}(z)} \frac{m(z_0,z)}{m(z_0)} P_t(z_0) \ , \ \frac{\partial }{\partial t} Q_t(z) = \sum_{z_1 \in \mathcal{F}(z)} \frac{m(z,z_1)}{m(z_1)} Q_t(z_1).$$
This allows us to write $\frac{\partial}{\partial t}f_t(z) = -\nabla g_t(z)$ where we define $$g_t(x_0x_1) := m(x_0,x_1) P_t(x_0)Q_{1-t}(x_1).$$ Similarly, defining $h_t(x_0x_1x_2) := m(x_0,x_1,_2) P_t(x_0)Q_{1-t}(x_2)$ we have $\frac{\partial}{\partial t} g_t(x_0x_1) = -\nabla h_t(x_0x_1)$. The positivity of $P_t$ and $Q_{1-t}$ implies the positivity of $g_t(x_0x_1)$. Moreover, the formula $$m(x_1) m(x_0,x_1,x_2) = m(x_0,x_1)m(x_1,x_2)$$ implies $$f_t(x_1) h_t(x_0x_1x_2) = g_t(x_0x_1)g_t(x_1x_2),$$ which shows that $(f_t)$ is a $W_{1,+}$-geodesic. 
\end{proof}

The task of finding a $W_{1,+}$-geodesic joining $f_0$ to $f_1$ is simplified by Theorem~\ref{thm:BinomBBC} because it turns it into the static problem of finding a coupling $\pi$ between $f_0$ and $f_1$ such that $\pi(x,y) := \frac{m(x,y)}{d(x,y)!} a(x) b(y) 1_{x\leq y}$ for a couple of functions $a(x)$, $b(y)$ defined on $G$ and for a function $m$ constructed in Definition~\ref{def:mextended}.

\medskip

This method can be used to prove the existence of $W_{1,+}$-geodesics with prescribed initial and final distributions:

\begin{thm}
Let $f_0, f_1 \in \mathcal{P}(G)$ be finitely supported. Then there exists a $W_{1,+}$-geodesic between $f_0$ and $f_1$.
\end{thm}

\begin{proof} Let $m:E(G) \rightarrow \mathbb{R}_+^*$ be any positive function with $\nabla m(x) =0$ for every $x \notin \mathcal{A},\mathcal{B}$, and extended to ordered families of $G$. We set $c(x,y) := \frac{m(x,y)}{d(x,y)!}$. By Theorem~\ref{thm:BinomBBC}, it suffices to prove the existence of a coupling $\pi \in \Pi_1(f_0,f_1)$ such that $\pi(x,y) = c(x,y)a(x)b(y) 1_{x \leq y}$ for a couple of positive $a,b : G \rightarrow \mathbb{R}$. 

\medskip

We will adopt the following point of view on the set $\Pi_1(f_0,f_1)$:

\medskip

Let $\mathcal{D} := \{(x,y) \in G \times G \ | \ x \leq y\}$. In the space $\mathbb{R}^\mathcal{D}$ with the usual sclar product, we consider the particular families of vectors $(j_{0,x})_{x \in G}$ and $(j_{1,y})_{y \in G}$ defined by
\begin{equation*}
\forall (x,y) \in \mathcal{D} \ , \ j_{0,x_0}(x,y) := 1_{x=x_0} \ , \ j_{1,y_0}(x,y) := 1_{y=y_0}.
\end{equation*}
If for every $(x,y) \in \mathcal{D}$ we have  $x_0 \neq x$ then we set $j_{0,x_0}=0$.

\medskip

If $\pi \in \mathbb{R}^\mathcal{D}$, we have 
\begin{equation*}
\pi \cdot j_{0,x_0} := \sum_{y \geq x_0} \pi(x_0,y) \ , \  \pi \cdot j_{1,y_0} := \sum_{x \leq y_0} \pi(x,y_0).
\end{equation*}

In particular, we have
\begin{equation*}
\Pi_1(f_0,f_1) := \mathbb{R}_+^{\mathcal{D}} \cap \left(\bigcap_{x_0 \in G} \{\pi : \pi \cdot j_{0,x_0} = f_0(x_0) \}\right) \cap \left( \bigcap_{y_0 \in G} \{\pi : \pi \cdot j_{1,y_0} = f_1(y_0) \} \right).
\end{equation*}

In other words, $\Pi_1(f_0,f_1)$ is seen as the intersection of the ``quadrant'' $\mathbb{R}_+^{\mathcal{D}}$ with an affine subspace of $\mathbb{R}^{\mathcal{D}}$ directed by the vector subspace $V^\perp$, where $V$ is the vector space generated by the families $(j_{0,x})_{x \in G}$ and $(j_{1,y})_{y \in G}$. 

\medskip

Depending on the dimension of $\Pi_1(f_0,f_1)$ as a subset of an affine subspace of $\mathbb{R}^{\mathcal{D}}$, we will consider two cases:

\begin{enumerate}
\item The dimension of $\Pi_1(f_0,f_1)$ is zero. In this case, the vector space $V$ is $\mathbb{R}^{\mathcal{D}}$. In particular, the vector $l \in \mathbb{R}^{\mathcal{D}}$, with components $l(x,y) := \frac{\pi(x,y)}{c(x,y)}$ for every couple $x\leq y \in \mathcal{D}$, can be written under the form $$l(x,y) = \sum_{x \in G} A(x) j_{0,x} + \sum_{y \in G} B(y) j_{1,y}$$ for a unique couple of functions $A,B$ defined on $G$. Considering the exponential of each side proves that $\pi$ can be written under the form $\pi(x, y) := c(x,y) a(x)b(y)1_{x \leq y} $ with $a(x) := \exp(A(x))$ and $b(y) := \exp(B(y))$.
\item The dimension of $\Pi_1(f_0,f_1)$ is positive. In this case we will use the fact that the interior $\Pi_1(f_0,f_1)$ is
non-empty and equal to the set of fully supported $W_1$-optimal couplings:
$$\Pi_1(f_0,f_1)^\circ = \{\pi \in \Pi_1(f_0,f_1) \ : \ \forall (x, y) \in \mathcal{D}, \pi(x, y) > 0\}.$$
The boundary of $\Pi_1(f_0,f_1)$ is thus described by:
$$\partial\Pi_1(f_0,f_1)= \{\pi \in \Pi_1(f_0,f_1) \ : \ \exists (x, y) \in \mathcal{D}, \pi(x, y) = 0\}.$$
\end{enumerate}

\medskip

We consider the mapping $J : \mathbb{R}_+^{\mathcal{D}} \rightarrow \mathbb{R}$ defined by 
\begin{equation} \label{eq:EntroFunct}
J(\pi) := \sum_{(x,y) \in \mathcal{D}} \pi(x,y) \log \left( \frac{\pi(x,y)}{c(x,y)} \right)-\pi(x,y),
\end{equation}
where the variables are denoted by $\pi(x,y)$, for $x\leq y$. The function $J$ is clearly continuous on $\mathbb{R}_+^{\mathcal{D}}$ and smooth on $\left(\mathbb{R}_+^* \right)^{\mathcal{D}}$. Moreover, we have:
\begin{equation} \label{eq:Jgradient}
\frac{\partial J}{\partial \pi(x,y)} = \log \left( \frac{\pi(x,y)}{c(x,y)} \right).
\end{equation}
The Hessian of $J$ is thus a diagonal matrix with positive coefficients $\left(\frac{1}{\pi(x,y)}\right)_{(x,y) \in \mathcal{D}}$, so $J$ is strictly convex on $\left(\mathbb{R}_+^* \right)^{\mathcal{D}}$.

\medskip

The set $\Pi_1(f_0,f_1)$ being compact, the infimum of $J$ on $\Pi_1(f_0,f_1)$ is attained for some coupling $\tilde{\pi}$. As $J$ is striclty convex and $\Pi_1(f_0,f_1)$ is a convex subset of $\mathbb{R}^\mathcal{D}$, we know that $\tilde{\pi}$ is unique and that we have either $\tilde{\pi} \in \partial \Pi_1(f_0,f_1)$ or $\tilde{\pi} \in \Pi_1(f_0,f_1)^\circ$ and in this second case $\tilde{\pi}$ is a critical point for the restriction to $\Pi_1(f_0,f_1)$ of the application $J$.

\medskip

Let us prove that $\tilde{\pi} \in \Pi_1(f_0,f_1)^\circ$: we consider a segment $\pi_t := (1-t)\pi_0+t\pi_1$, where $\pi_0 \in \partial \Pi_1(f_0,f_1)$ and $\pi_1 \in \Pi_1(f_0,f_1)^\circ$. Each $\pi_t$ is in $\Pi_1(f_0,f_1)$, by  convexity. The function $J(t) := J(\pi_t)$ is continuous on $[0,1]$, smooth on $]0,1[$ and we have:
\begin{equation*} 
J'(t) = \sum_{(x,y) \in \mathcal{D}} (\pi_1(x,y)-\pi_0(x,y)) \log \left( \pi_t(x,y) \frac{d(x,y)!}{m(x,y)} \right).
\end{equation*}
As $\pi_0 \in \partial \Pi_+$, there exists $(x_0,y_0) \in \mathcal{D}$ such that $\pi_0(x_0,y_0)=0$ and we have
\begin{equation*}
\lim_{t \rightarrow 0} (\pi_1(x_0,y_0)-\pi_0(x_0,y_0)) \log \left( \pi_t(x_0,y_0) \frac{d(x_0,y_0)!}{m(x_0,y_0)} \right)= - \infty,
\end{equation*}
so we have $\lim_{t \rightarrow 0} J'(t) = -\infty$. The infimum of $J$ on $\Pi_1(f_0,f_1)$ is thus not attained on $\partial \Pi_1(f_0,f_1)$.

\medskip

We have proven the existence of a unique critical point $\tilde{\pi} \in \Pi(f_0,f_1)^\circ$ for the restriction to $\Pi_1(f_0,f_1)$ of $J$. As $\Pi_1(f_0,f_1)$ is a subset of an affine space directed by a vector subspace $V^\perp$, we know that $$\grad_{\tilde{\pi}} J \in V.$$ In other terms, 
\begin{equation} \label{eq:PiPlusLagrange}
\grad_{\pi_0}(J) = \sum_{x \in G} A(x) j_{0,x} + \sum_{y \in G} B(y) j_{1,y} 
\end{equation}
for a couple of functions $A,B : G \rightarrow \mathbb{R}$. Due to the particular form taken by $j_{0,x}$ and $j_{1,y}$, Equation~\eqref{eq:PiPlusLagrange} can be rewritten in a simple way: $$\forall (x,y) \in \mathcal{D} \ , \ \grad_{\pi_0}(J)(x,y) = A(x)+B(y).$$
But equation~\eqref{eq:Jgradient} gives an explixcit formula for $\grad_{\pi_0}(J)(x,y)$, which allows us to write, for $(x,y) \in \mathcal{D}$:
\begin{eqnarray*}
\frac{\tilde{\pi}(x,y)}{c(x,y) } &=& \exp \left(\grad_{\pi_0}(J)(x,y)\right) \\
 &=& \exp(A(x)+B(y)) = a(x)b(y),
\end{eqnarray*}
where $a(x) := \exp(A(x))$ and $b(y) := \exp(B(y))$. Theorem~\ref{thm:BinomBBC} gives the a $W_{1,+}$-geodesic $(f_t)$ between $f_0$ and $f_1$ constructed from the function $m$ and the coupling $\tilde{\pi}$
\end{proof}

\begin{remark}
The particular form taken by $W_{1,+}$-geodesics (see Equation~\eqref{eq:ftgeneralform}) and the minimisation problems associated by the functionals~\eqref{eq:EntroFunct} and~\eqref{eq:W1plusBB}, are reminiscent of the theory of Entropic Interpolations, constructed in a recent series of articles by L\'{e}onard. A survey of the main results of this theory is found in~\cite{LeonardSurvey}. A construction of entropic interpolations and a discussion of the cases where they can be described as mixtures of binomials is found in~\cite{LeoLazy}. Another paper, see~\cite{LeoConvex}, addresses the question of the convexity of entropy along such interpolations. 

\medskip

A major difference between these two kinds of interpolations lies in their construction: in order to define an entropic interpolation on a graph $G$, one requires an underlying Markov chain to which is canonically associated a positive measure $R^{01}$ on the set of couples of vertices $x,y \in G$. On the other hand, the definition of a $W_{1,+}$-interpolation does not require an underlying Markov chain. It only relies on the ``metric-measure'' properties of the graph $G$, endowed with its counting measure.  However, to each $W_{1,+}$-geodesic is associated a function $m$ on the ordered subsets of $G$, which is used to construct sub-Markov kernels.

\medskip

A complete understanding of the links between entropic interpolations and $W_{1,+}$-geodesics, and more especially between the measure $R^{01}$ of entropic interpolations and the function $m$ of $W_{1,+}$-geodesics, is still under investigation.
\end{remark}

\end{document}